%% file: main.tex
\pdfoutput = 1
\documentclass[journal]{IEEEtran}

\makeatletter
\@ifclasslater{IEEEtran}{2012/12/26}
{\newcommand{\killproofname}{\unskip\nopunct}}
{\newcommand{\killproofname}[1]{\unskip\aftergroup\ignorespaces\ignorespaces}}
\makeatother

\input{MZ_pack}

\begin{document}
	
\title{Distributed Online Convex Optimization with Improved Dynamic Regret}

\author{Yan Zhang,~\IEEEmembership{Student Member,~IEEE,}
	Robert J. Ravier~\IEEEmembership{Member,~IEEE,} Vahid~Tarokh,~\IEEEmembership{Fellow,~IEEE}, and  \\Michael M. Zavlanos,~\IEEEmembership{Senior Member,~IEEE}
	\thanks{Yan Zhang and Michael M. Zavlanos are with the Department
		of Mechanical Engineering and Material Science, Duke University, Durham,
		NC, 27708, USA e-mail: \{yan.zhang2, michael.zavlanos\}@duke.edu.Robert J. Ravier and Vahid Tarokh are with Department of Electrical and Computer Engineering, Duke University, Durham, NC, 27708, USA e-mail: \{robert.ravier, vahid.tarokh\}@duke.edu. This work is supported in part by AFOSR under award number FA9550-19-1-0169 and by DARPA under grant number FA8650-18-1-7837.}}

\markboth{IEEE TRANSACTIONS ON AUTOMATIC CONTROL}%
{Shell \MakeLowercase{\textit{et al.}}: Bare Demo of IEEEtran.cls for IEEE Journals}

\maketitle

\begin{abstract}
In this paper, we consider the problem of distributed online convex optimization, where a group of agents collaborate to track the global minimizers of a sum of time-varying objective functions in an online manner. 
Specifically, we propose a novel distributed online gradient descent algorithm that relies on an online adaptation of the gradient tracking technique used in static optimization. We show that the dynamic regret bound of this algorithm has no explicit dependence on the time horizon and, therefore, can be tighter than existing bounds especially for problems with long horizons. Our bound depends on a new regularity measure that quantifies the total change in the gradients at the optimal points at each time instant. Furthermore, when the optimizer is approximatly subject to linear dynamics, we show that the dynamic regret bound can be further tightened by replacing the regularity measure that captures the path length of the optimizer with the accumulated prediction errors, which can be much lower in this special case. We present numerical experiments to corroborate our theoretical results.
\end{abstract}


\section{Introduction}
\input{tex/Intro}

\section{Preliminaries and Problem Definition}
\input{tex/Preliminaries}

\section{Algorithm}
\input{tex/Algorithm}

\section{Convergence Analysis}
\input{tex/Convergence}

\section{Numerical Experiments}
\input{tex/Experiements}

\section{Conclusions}
\input{tex/Conclusion}

\bibliographystyle{IEEEtran}
\bibliography{biblio}

\section*{Appendix}
\input{tex/Appendix}

\begin{IEEEbiography}[{\includegraphics[width=1in,height=1.25in,clip,keepaspectratio]{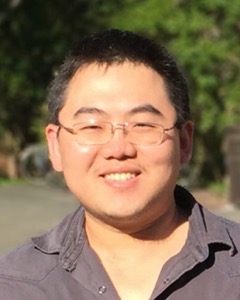}}]{Yan Zhang (S’16)}
received his bachelor’s degree in mechanical engineering from Tsinghua University, Beijing, China in 2014, and his master's degree in mechanical engineering from Duke University, Durham, NC in 2016. Currently, he is pursuing a doctoral degree in mechanical engineering at Duke University, Durham, NC.
His research interests include distributed optimization and distributed reinforcement learning algorithms.
\end{IEEEbiography}

\begin{IEEEbiography}[{\includegraphics[width=1in,height=1.25in,clip,keepaspectratio]{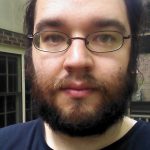}}]{Robert J. Ravier (M’19)}
graduated summa cum laude with a B.A. in Mathematics from Cornell University in 2013 under the supervision of Robert Strichartz. He graduated Duke University with a Ph.D. in Mathematics in 2018 under the supervision of Ingrid Daubechies. During his studies, he received the Kieval Prize for outstanding student in mathematics in 2013 and was awarded an NDSEG Graduate Research Fellowship funded by AFOSR in 2015.

He is currently a Postdoctoral Associate under the supervision of Vahid Tarokh. His research interests stem a wide number of fields and is mostly focused on real world applications of signal processing, optimization, dynamical systems, and statistics. He is the author of the SAMS software package for biological surface analysis, and has assisted in developing quantitative methodologies for partisan gerrymandering that have been at the center of multiple U.S. state and federal court cases, including U.S. Supreme Court cases Gil v. Whitford and Rucho v. Common Cause.
\end{IEEEbiography}

\begin{IEEEbiography}[{\includegraphics[width=1in,height=1.25in,clip,keepaspectratio]{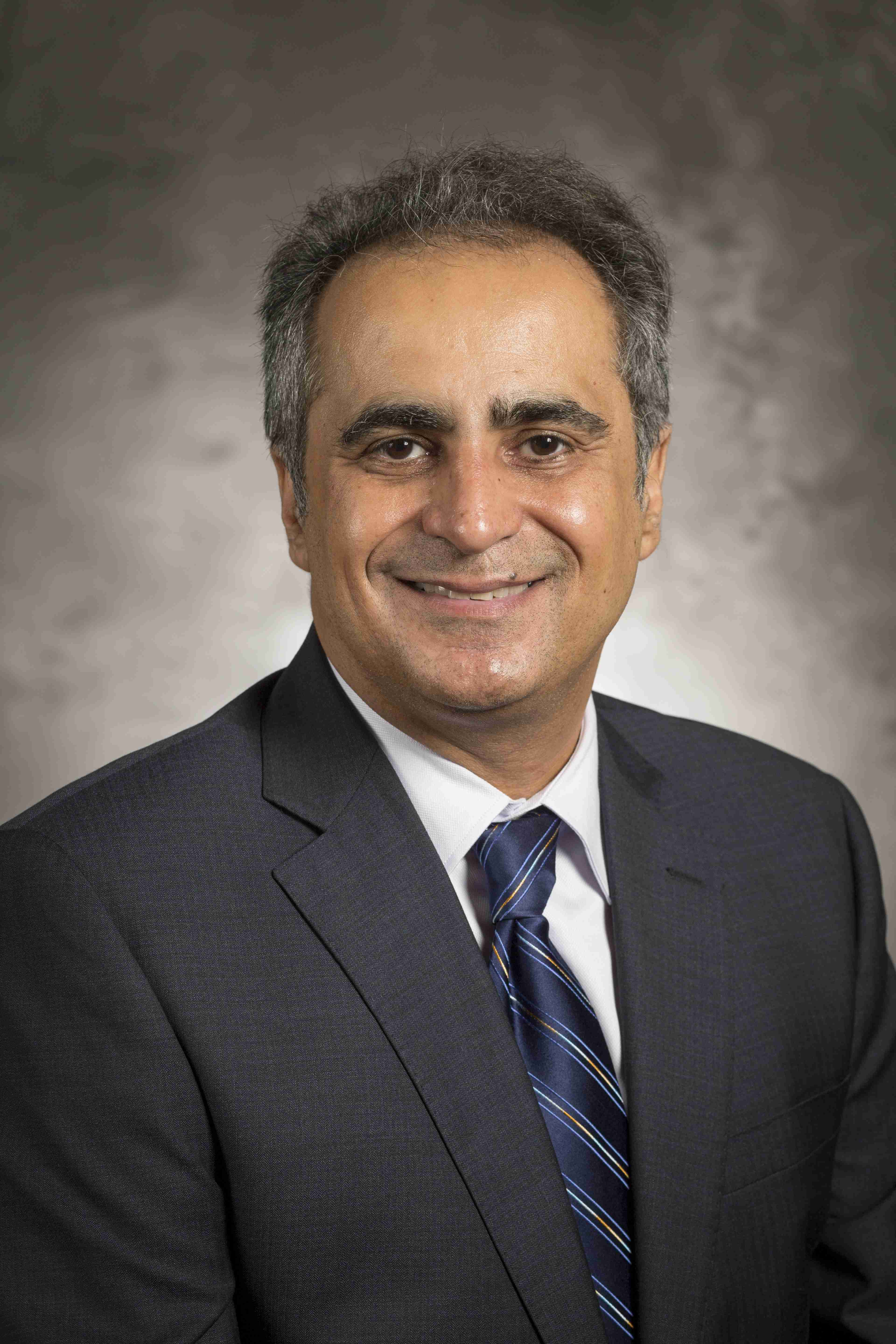}}]{Vahid Tarokh (F’09)}
worked at AT\&T Labs-Research and AT\&T Wireless Services until August 2000 as Member, Principal Member of Technical Staff and, finally, as the Head of the Department of Wireless Communications and Signal Processing. In September 2000, he joined the Massachusetts Institute of Technology (MIT) as an Associate Professor of Electrical Engineering and Computer Science. In June 2002, he joined Harvard University as a Gordon McKay Professor of Electrical Engineering and Hammond Vinton Hayes Senior Research Fellow. He was named Perkins Professor of Applied Mathematics, and Hammond Vinton Hayes Senior Research Fellow of Electrical Engineering in 2005. In Jan 2018, He joined Duke University, as the Rhodes Family Professor of Electrical and Computer Engineering, Bass Connections Endowed Professor, and Professor of Computer Science, and Mathematics. From Jan 2018 to May 2018, He was also a Gordon Moore Distinguished Scholar in the California Institute of Technology (CALTECH).  Since Jan 2019, he has also been named as a Microsoft Data Science Investigator at Duke University.

He has supervised 33 Post-doctoral Fellow and 16 PhD students; over 55\% of these are Professors at Research Universities, and the rest are research scientists at various US Government labs (Lincoln Labs, NASA), and industrial labs. In addition to these, he has supervised 12 M.S. thesis, one undergraduate thesis, and 3 M.S. non-thesis students. In the summer of 2016, Dr. Tarokh supervised 6 High School Summer Student Research on development of tactile gloves and applications, under a US Army HSAP Program.
\end{IEEEbiography}

\begin{IEEEbiography}[{\includegraphics[width=1in,height=1.25in,clip,keepaspectratio]{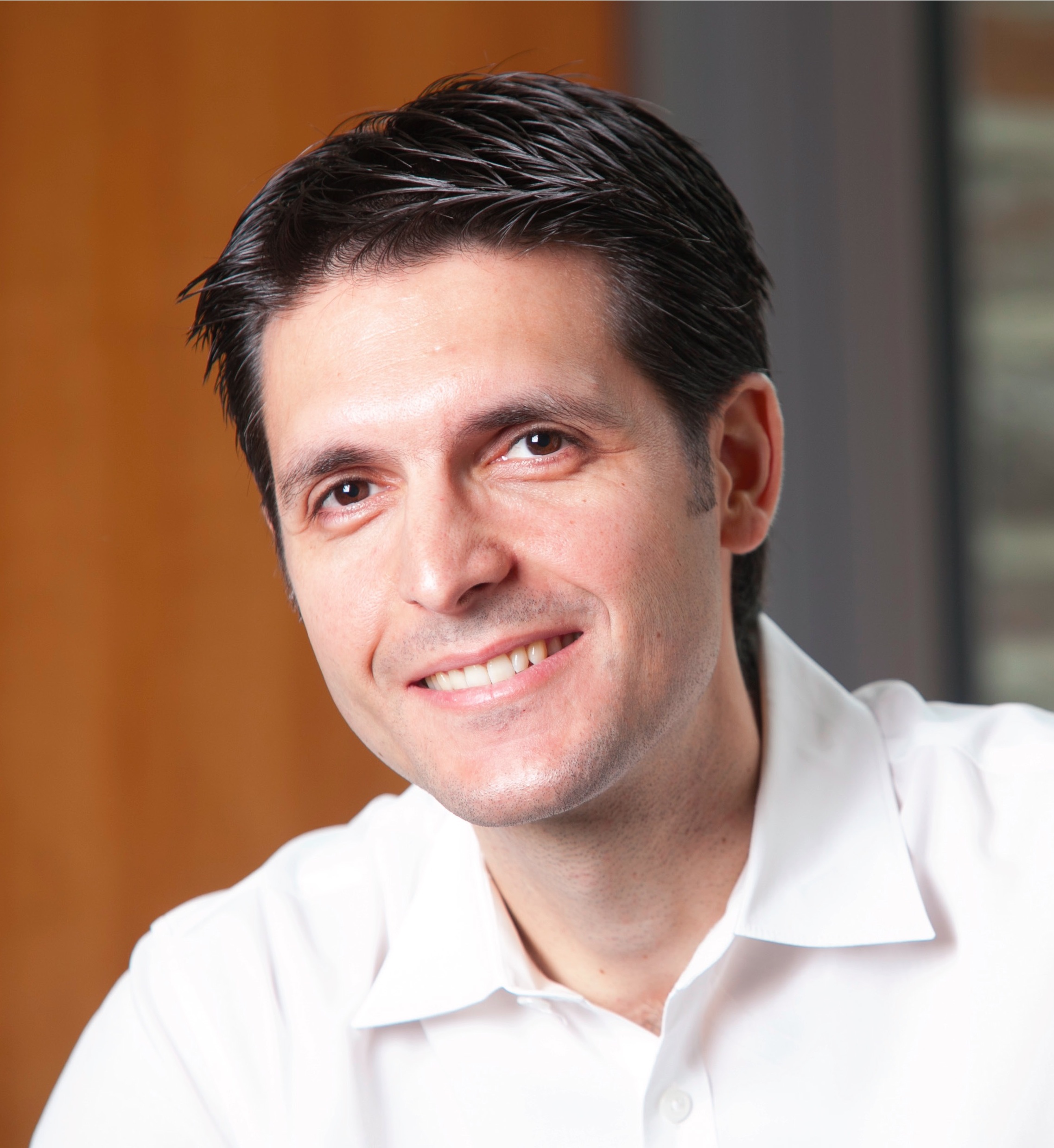}}]{Michael M. Zavlanos (S'05M'09SM'19)}
received the Diploma in mechanical engineering from the National Technical University of Athens (NTUA), Athens, Greece, in 2002, and the M.S.E. and Ph.D. degrees in electrical and systems engineering from the University of Pennsylvania, Philadelphia, PA, in 2005 and 2008, respectively.

He is currently an Associate Professor in the Department of Mechanical Engineering and Materials Science at Duke University, Durham, NC. He also holds a secondary appointment in the Department of Electrical and Computer Engineering and the Department of Computer Science. Prior to joining Duke University, Dr. Zavlanos was an Assistant Professor in the Department of Mechanical Engineering at Stevens Institute of Technology, Hoboken, NJ, and a Postdoctoral Researcher in the GRASP Lab, University of Pennsylvania, Philadelphia, PA. His research focuses on control theory and robotics and, in particular, networked control systems, distributed robotics, cyber-physical systems, and learning for control.

Dr. Zavlanos is a recipient of various awards including the 2014 Naval Research Young Investigator Program (YIP) Award and the 2011 National Science Foundation Faculty Early Career Development (CAREER) Award.
\end{IEEEbiography}
\vfill
%
%
%
%




\end{document}

%% file: MZ_pack.tex
\usepackage{amsmath}
\usepackage{amssymb}
\usepackage{amsthm}

\newtheorem{thm}{Theorem}[section]

\newtheorem{lem}[thm]{Lemma}
\newtheorem{asmp}[thm]{Assumption}

\newtheorem{exmp}{Example}
\newtheorem{defn}[thm]{Definition}
\newtheorem{rem}[thm]{Remark}

\newcounter{mycounter}

\usepackage{stmaryrd}
\usepackage{graphicx} 
\usepackage{bbm}
\usepackage{epsfig} 

\usepackage{epstopdf}
\usepackage{cite}
\usepackage[noend,ruled,linesnumbered,resetcount]{algorithm2e}
\SetKwComment{Comment}{$\triangleright$\ } {}
\usepackage{multirow}
\usepackage{rotating}
\usepackage{subfigure}
\usepackage{color}
\usepackage{mysymbol}
\usepackage{url}

\usepackage{tabularx,ragged2e,booktabs}
\usepackage[normalem]{ulem}
\usepackage[utf8x]{inputenc}
\usepackage{multirow}

\newcommand{\RNum}[1]{\uppercase\expandafter{\romannumeral #1\relax}}

\usepackage[hang,flushmargin]{footmisc}
\makeatletter
\newcommand{\algorithmfootnote}[2][\footnotesize]{%
  \let\old@algocf@finish\@algocf@finish
  \def\@algocf@finish{\old@algocf@finish
    \leavevmode\rlap{\begin{minipage}{\linewidth}
    #1#2
    \end{minipage}}%
  }%
}
\makeatother

  \makeatletter
  \setbox0\hbox{$\xdef\scriptratio{\strip@pt\dimexpr
    \numexpr(\sf@size*65536)/\f@size sp}$}
\newcommand{\scriptveryshortarrow}[1][3pt]{{%
    \hbox{\rule[\scriptratio\dimexpr\fontdimen22\textfont2-.2pt\relax]
               {\scriptratio\dimexpr#1\relax}{\scriptratio\dimexpr.4pt\relax}}%
   \mkern-4mu\hbox{\let\f@size\sf@size\usefont{U}{lasy}{m}{n}\symbol{41}}}}

\makeatother

%% file: tex/Intro.tex
Distributed optimization has recently received considerable attention, particularly due to its wide applicability in the areas of control and learning \cite{lee2017distributed, rabbat2004distributed, nedic2017fast}. The goal is to decompose large optimization problems into smaller, more manageable subproblems that are solved iteratively and in parallel by a group of communicating agents. As such, distributed algorithms avoid the cost and fragility associated with centralized coordination, and provide better privacy for the autonomous decision makers. Popular distributed optimization methods in the literature include distributed subgradient methods \cite{nedic2009distributed, lee2017approximate}, dual averaging methods \cite{duchi2011dual}, and augmented Lagrangian methods \cite{chatzipanagiotis2015augmented, chatzipanagiotis2015distributed, chatzipanagiotis2017convergence, zhang2018consensus}.

The above distributed optimization methods usually assume a static objective function. Nevertheless, in practice, objectives can be time-varying. Time-varying objectives frequently appear in online learning, where newly observed data result in new objectives to minimize, and in distributed tracking, where the objective is to accurately follow the time-varying states of the targets of interest (e.g. positions and velocities) \cite{duchi2011adaptive, shahrampour2018distributed}.
These problems can be solved using online optimization algorithms that update the decisions using real-time streaming data, in contrast to their off-line counterparts that first collect problem data and then use them for decision making.

The performance of online optimization algorithms is typically measured using notions of regret. Depending on the problem setting, different notions of regret have been proposed in the literature. For example, static regret, which measures the additional loss caused by the online optimization algorithm compared to the offline optimizer assuming all loss functions are known in hindsight, is used when the parameter that is estimated is assumed to be time invariant, as in online learning \cite{duchi2011adaptive}. 
The static regret of online gradient descent algorithms has been extensively studied in the literature; see, e.g., \cite{duchi2011adaptive, zinkevich2003online,hazan2007logarithmic, chiang2012online}.
For general convex problems, it has been shown that a sublinear regret rate $O(\sqrt{T})$ can be achieved \cite{zinkevich2003online}, which can be improved to $O(\log(T))$ assuming strong convexity \cite{hazan2007logarithmic}. The work in \cite{tatarenko2018minimizing} extends these results to zeroth-order methods. Unconstrained distributed online gradient descent algorithms are studied in \cite{hosseini2016online, mateos2014distributed, akbari2015distributed, cesa2019cooperative}. These methods deal with unconstrained problems and still achieve sublinear regret rates provided the stepsizes are chosen appropriately and the network of agents is connected. 
To handle constrained online optimization problems, the approaches in \cite{lee2017sublinear,yuan2017adaptive,paternain2019distributed} employ distributed online saddle point algorithms and show that they achieve the same regret rate as for unconstrained problems.

\begin{table*}[!htb]
	\caption{Bounds on the Dynamic Regret of Distributed Online Optimization Algorithms \protect \footnotemark } \label{tab:DynamicRegretBounds}
	\setlength\tabcolsep{0pt} 
	\begin{small}
		\begin{tabular*}{\textwidth}{@{\extracolsep{\fill}}ccccc}
			\toprule[.15em]
			& Problem setting & Number of Updates per Sample & Upper Bound on Stepsize & Regret Bound \\
			\midrule[.05em]
			\cite{shahrampour2018distributed} & Constrained & 1 & $O(\sqrt{\frac{\mathcal{P}_T^A}{T}})$ & $O(\sqrt{T (1 + \mathcal{P}_T^A)})$ \\ 
			\midrule[.05em]
			\cite{lu2019online} & Constrained & 1 & $O(\frac{1}{t})$ & $O( \sqrt{1 + \mathcal{P}_T} T^{3/4} \sqrt{\ln T})$ \\ 			
			\midrule[.05em]
			\cite{zhao2019decentralized} & Unconstrained & 1 & $O(\sqrt{\frac{\mathcal{P}_T}{T}})$ & $O(\sqrt{T (1 + \mathcal{P}_T)}+ \mathcal{P}_T)$ \\ 			
			\midrule[.05em]
			\cite{nazari2019dadam} & Constrained & 1 & $O(\frac{1}{t})$ & $O(\sqrt{1 + \log T} G_T + \sqrt{T} \mathcal{P}_T)$ \\ 	
			\midrule[.05em]
			\cite{dixit2019online_b} & Constrained & $K$ & $O(1)$ & $O(\log(T) (1 + \mathcal{P}_T))$ \\ 	
			\midrule[.05em]
			This work & Unconstrained & $1$ & $O(1)$ & $O( 1 + \mathcal{P}_T^A + \mathcal{V}_T^A)$ \\							
			\bottomrule[.15em]
		\end{tabular*}
	\end{small}
\end{table*}

Dynamic regret is a more appropriate performance measure when the underlying parameter of interest is time-varying. Dynamic regret compares the loss incurred by the online algorithm to the optimal loss incurred by the sequence of optimizers that minimize the objective functions at each time step separately. The dynamic regret of centralized online gradient descent algorithms is studied in \cite{zinkevich2003online, hall2015online,besbes2015non, jadbabaie2015online,dixit2019online}. In contrast to online optimization problems with time invariant parameters and static regret methods, sublinear dynamic regret rates $O(\sqrt{T})$ can not be achieved here; rather, the growth of dynamic regret depends on the regularity measures associated with the time-varying problem \cite{besbes2015non}. 
These measures can be related to the rate of change of the function values or minimizers over time \cite{jadbabaie2015online}
\begin{equation*}
\mathcal{P}_{T}:= \sum_{t=0}^{T} \|{x_{t+1}^{\ast} - x_{t}^\ast}\|.
\end{equation*}
Dynamic regret methods for constrained online optimization problems are studied in \cite{chen2017online}. All these methods focus on centralized optimization problems. Distributed online gradient descent algorithms for unconstrained or set-constrained problems are proposed in \cite{shahrampour2018distributed,lu2019online, zhao2019decentralized,nazari2019dadam,dixit2019online_b}, while \cite{yi2019distributed} analyzes the dynamic regret of time-varying problems with coupled explicit contraint functions.

\footnotetext{
In Table~\ref{tab:DynamicRegretBounds}, the constants hidden in the big $O$ notation only depend on the Lipschitz and strong convexity constants of the objective functions, the network mixing rate parameter $\sigma_W$ and the size of the network $n$, which are all independent of the problem horizon $T$ and regularity measures $\mathcal{P}_T$ or $\mathcal{P}_T^A$. The measure $\mathcal{P}_T^A$ extends $P_T$ to the case where the optimizers of the time-varying objective functions are known to follow noisy linear dynamics, and it is later defined in \eqref{eqn:PathLength}. It can be reduced to the measure $\mathcal{P}_T$ when the matrix $A$ is identity. The measure $\mathcal{V}_T^A$ that appears in our regret bounds in Table~\ref{tab:DynamicRegretBounds} is defined similarly.}
The dynamic regret bounds provided in the above literature are listed in Table~\ref{tab:DynamicRegretBounds}.
Note that all existing bounds explicitly depend on the problem horizon $T$, which makes them loose especially for large horizons. To see this, consider the special case of a static problem where $\mathcal{P}_T = 0$. Then, the regret bounds in above works still grow with respect to $T$. 
It is thus of theoretical and practical interest to investigate whether these regret bounds can be tightened by removing their dependency on problem horizon $T$. The works in \cite{mokhtari2016online,yang2016tracking,ajalloeian2020inexact} have shown that this is possible for centralized problems assuming strong convexity of the objective funtion. 
Similarly, \cite{zhang2017improved} removes dependency of the dynamic regret on the time horizon under less restrictive assumptions.
For distributed problems, the dependence on $T$ was removed in \cite{ZhangCDC2019_OnlineOptimization} provided that the sum of the local objective functions was strongly convex, though the regret bound achieved depends on the gradient path-length regularity measured by the sup-norm of the time difference of gradients
\begin{equation*}
\mathcal{V}_{T,\sup} := \sum_{t=0}^{T} \max_{x \in \mathbb{R}^d} || g_{t+1}(\mathbf{1} \otimes x) - g_{t}(\mathbf{1} \otimes x) ||,
\end{equation*}
similar to \cite{chiang2012online,jadbabaie2015online, dixit2019online}, where $g_t$ denotes the gradients of the local objective functions at time $t$. It is clear that this term can become quite large, and it is thus of interest to explore whether the regret bound can be further improved.
In this work, we propose a novel distributed online gradient algorithm that employs gradient tracking \cite{qu2018harnessing,pu2020distributed,shi2015extra,xi2017dextra}, as in \cite{ZhangCDC2019_OnlineOptimization}, and show that the dynamic regret of our algorithm can be bounded without explicit dependence on $T$ and with the regularity measure that depends on the change of the gradients at the optimal points

\begin{equation*}
\mathcal{V}_T := \sum_{t=0}^{T} || g_{t+1}(\mathbf{1} \otimes x_{t}^\ast) - g_{t}(\mathbf{1} \otimes x_{t}^\ast) ||.
\end{equation*}
It is clear that this regularity measure is tighter than $\mathcal{V}_{T,\sup}$ used in\cite{ZhangCDC2019_OnlineOptimization}. 
Comparing to the regret bounds in \cite{shahrampour2018distributed,lu2019online, zhao2019decentralized,nazari2019dadam,dixit2019online_b}, our bound is tighter when $\mathcal{V}_T$ is in the same order as $\mathcal{P}_T$, $\mathcal{P}_T$ 
is in the order $o(T)$, and the horizon $T$ is long.
In addition, the dynamic regret bound of our proposed algorithm can be achieved by selecting the stepsize without knowledge of the problem horizon $T$ and regularity measure $\mathcal{P}_T$, in contrast to the stepsize rules in \cite{shahrampour2018distributed,zhao2019decentralized}. And it only requires a single update per sample of the time-varying objective function, while $K$ updates per sample are required in \cite{dixit2019online_b}.
 

The performance of our proposed algorithm can be further improved in the special case where the optimal points of the time-varying objective functions follow a noisy linear dynamical system, as in \cite{shahrampour2018distributed}. If an estimator of this dynamical system is available, it can be directly incorporated in our algorithm to improve the dynamic regret bounds, provided that this estimator is sufficiently accurate. 
Specifically, we show that when the optimizers are subject to known linear dynamics, the proposed prediction step allows to reduce the regularity measuring the path-length of the optimizers to the prediction error, similar to the analysis in \cite{chen2016using, ravier2019prediction}.  
Knowledge of the optimizer dynamics is generally not possible in adversarial settings as
in \cite{bubeck2012regret}, where an adversary agent selects the objective function to reveal to the agent at the next time step, but is it possible, e.g., in distributed tracking problems as those considered in \cite{shahrampour2018distributed,simonetto2017decentralized}.
Specifically, in \cite{simonetto2017decentralized}, a distributed prediction-correction scheme is studied, where the time-varying objective functions  evolve in continuous time and the prediction steps are conducted using second-order derivative information. The tracking performance guarantees in \cite{simonetto2017decentralized} require bounded third-order derivatives of the objective functions.
In contrast, we do not require second-order information to conduct prediction, and our analysis is based on weaker assumptions. Specifically, we only require that the change of the objective functions over time is bounded, which can be satisfied even when the time derivative of the objective function does not exist.

The rest of this paper is organized as follows. Section II formulates the distributed online optimization problem under consideration. Section III presents the proposed algorithm and develops theoretical results that characterize its dynamic regret. We present numerical experiments in Section IV and make concluding remarks in Section V.


%% file: tex/Preliminaries.tex
\subsection{Problem Definition}

We consider the time-varying optimization problem:

\begin{equation} \label{eqn:ProblemForm}
	\min_{x_t} f_t(x_t),
\end{equation}
where $x_{t} \in \mathbb{R}^d$, and $f_t(x_t) = \sum_{i=1}^{n} f_{i,t}(x_t)$ is a sum of local loss functions $f_{i,t}(x_t)$ assigned to a group of $n$ agents that are tasked with solving problem~\eqref{eqn:ProblemForm}. The loss functions $f_{i,t}$ are assumed to be time-varying and are not revealed to the agents until each agent has made its decision $x_{i,t} \in \mathbb{R}^d$ at time $t$. The local loss function $f_{i,t}(x_{i,t})$ can only be observed by agent $i$, thus requiring communication among the agents in order to solve problem~\eqref{eqn:ProblemForm}.
Due to the fact that the objective function $f_{i,t}(x)$ is revealed to the agents in an online and distributed fashion, and only one update is allowed between each sample of the time-varying objective function, it is unlikely to find the exact optimizer $x_t^\ast := \arg \min_{x_t \in \mathbb{R}^d} f_t(x_t)$ at each time instant. Instead, it is desriable to design an online distributed optimization algorithm to track the time-varying optimizer $x_t^\ast$, so that each agent can gradually find local estimates $x_{i,t}$ so that $x_{i,t}$ (or $f_t(x_{i,t})$) is close to $x_t^\ast$ (or $f_t(x_{t}^\ast)$).

\subsection{Regret Measures and Assumptions}
In online optimization, the main performance measures of interest over a time horizon $0 \leq t \leq T$ are notions of regret.
Two regret measures commonly considered in the literature are static regret and dynamic regret. 

\begin{defn} (Static Regret)
	Given the sequence of the local decisions $x_{i,t}$ of an online distributed optimization algorithm, the static regret is defined as
	
	\begin{equation}\label{eqn:static}
	R_{T}^{s} := \frac{1}{n} \sum_{i=1}^{n} \sum_{t=0}^{T} \big( f_{t}(x_{i,t}) - \min_{x} \sum_{t=0}^{T} f_{t}(x) \big)
	\end{equation}
	\noindent and captures the performance of the algorithm with respect to the best fixed-point in hindsight.
\end{defn}

In this work, we are primarily interested in dynamic regret.

\begin{defn} (Dynamic Regret)
	Given the sequence of the local decisions $x_{i,t}$ of an online distributed optimization algorithm, the dynamic regret is defined as
	
	\begin{equation}\label{eqn:dynamic}
		R_{T}^{d} := \frac{1}{n} \sum_{i=1}^{n} \sum_{t=0}^{T} f_{t}(x_{i,t}) - \sum_{t=0}^{T} \min_{x_{t}}  f_{t}(x_{t})
	\end{equation}
	\noindent that captures the performance of the algorithm with respect to the optimal loss induced by the time-varying sequence of optimizers $\{x_t^\ast\}$.
\end{defn}
In constrast to the static regret, where sublinear regret rates $O(\sqrt{T})$ or $O(\log(T))$ can be achieved in \cite{duchi2011adaptive, zinkevich2003online,hazan2007logarithmic, chiang2012online}, the upper bounds for the dynamic regret do not only depend on the horizon $T$, but also on certain regularity properties of the time-varying problem \cite{besbes2015non}.
We first make the following assumption on the underlying dynamics of the time-varying problem.
\begin{asmp} \label{assum:DynamicalMatrix}
	The sequence of optimizers $x_{t}^{\ast}$ is subject to the following stable, noisy linear dynamical system
	
	\begin{equation}\label{eqn:LinearDynamics}
	x_{t+1}^\ast = A x_t^\ast + w_t,
	\end{equation}
	\noindent where $\|A\| \leq 1$ and $w_{t}$ is a noise term uniformly bounded by a constant $C_w$ over time, i.e.,
	
	\begin{equation}\label{eqn:BoundNoise}
	\|w_t\| \leq C_{w}, \text{ for all } t. 
	\end{equation}
\end{asmp}
We note here that Assumption~\ref{assum:DynamicalMatrix} is general enough to also capture the case where the optimizers are subject to unknown nonlinear dynamics. In this case, we can simply set $A$ to be the identity matrix and $w_t = x_{t+1}^\ast - x_t^\ast$ so that Assumption~\ref{assum:DynamicalMatrix} is satisfied. 
According to Assumption~\ref{assum:DynamicalMatrix}, when an estimate of the dynamical matrix $A$ is available, we can predict the future optimizer up to a prediction error $w_t$, whose magnitude is much smaller than $\|x_{t+1}^\ast - x_t^\ast\|$. Later in Theorem~\ref{thm:BoundRegret}, we show that in this case we can achieve a lower regret bounded by the accumulated prediction errors $\|w_t\|$ rather than the path length of the optimizers $\|x_{t+1}^\ast - x_t^\ast\|$. Now let $g_t([x_{1,t}^T, \dots, x_{n,t}^T]^T) = [\nabla f_{1,t}(x_{1,t})^T, \dots, \nabla f_{n,t}(x_{n,t})^T]^T$. In what follows, we also assume that the variation of the gradients at the optimizers is uniformly bounded over time.
	
\begin{asmp} \label{assum:BoundedTimeVaryingFunction}
	There exists a constant $C_g$ such that the objective functions gradients satisfy
	
	\begin{equation}\label{eqn:BoundedChangeGradient}
	\|g_{t+1}(\mathbf{1}\otimes Ax_t^\ast) - g_{t}(\mathbf{1}\otimes x_t^\ast)\| \leq C_g,
	\end{equation}
	\noindent for all $t$. In addition, the gradients of local objective functions $f_{i,t}(x)$ are uniformly bounded at $x_t^\ast$ for all time $t$. 
\end{asmp}	

Given the above assumptions on the time-varying problem, we consider regularity measures: the path length of the optimizer with prediction
\begin{equation} \label{eqn:PathLength}
\mathcal{P}_{T}^A:= \sum_{t=0}^{T} \|{x_{t+1}^{\ast} - Ax_{t}^\ast}\|, 
\end{equation}
and the path length of the gradient variation

\begin{equation} \label{eqn:GradPathLength}
\mathcal{V}_T^A := \sum_{t=0}^{T} || g_{t+1}(\mathbf{1} \otimes Ax_{t}^\ast) - g_{t}(\mathbf{1} \otimes x_{t}^\ast) ||.
\end{equation}
The measure $\mathcal{P}_{T}^A$ is a generalization of the commonly used regularity measure in \cite{zinkevich2003online}, the path length of the optimizer. To see this, when no prediction is used, i.e. $A$ is the identity matrix, $\mathcal{P}_{T}^A$ is reduced to the path length traveled by the optimizer, same as in \cite{zinkevich2003online}. When prediction is used, $\mathcal{P}_{T}^A$ accumulates the prediction error. The gradient variation measure we propose in \eqref{eqn:GradPathLength} is novel and different from the one based on the sup-norm considered in \cite{chiang2012online, jadbabaie2015online, dixit2019online,ZhangCDC2019_OnlineOptimization}:

\begin{equation}
\mathcal{V}_T := \sum_{t=0}^{T} \max_{x \in \mathbb{R}^d} || g_{t+1}(\mathbf{1} \otimes x) - g_{t}(\mathbf{1} \otimes x) ||.
\end{equation}
It is easy to see that the individual terms in the sum, and hence the entire quantity, can become arbitrarily large in general; the case of a quadratic objective function provides a natural example.

In the remainder of the paper, we make the following standard assumptions on the objective functions and their gradients.
%
\begin{asmp} \label{assum:LipschitzGradient}
	For all $i$ and $t,$ the gradient of the function $f_{i,t}$ is $L_{g}$-smooth, i.e. there exists a constant $L_{g} > 0$ such that 
	
	\begin{equation*}
	\| \nabla f_{i,t} (x) - \nabla  f_{i,t} (y) \| \leq L_g \|x - y\|, \text{ for all } x, y.
	\end{equation*}
\end{asmp}
As in \cite{mokhtari2016online}, we make the following assumption on strong convexity of the objective functions $f_t$.
\begin{asmp} \label{assum:StrongConvexity}
	For all $t,$ the function $f_{t}$ is $\mu$-strongly convex, i.e. there exists a constant $\mu > 0$ such that, for all $x$ and $y,$ we have:
	
	\begin{equation*}
	\nabla f_{t} (y) \geq  f_{t}(x) + \nabla f_{t}(x)^{T}(y-x) + \frac{\mu}{2} \| y-x \|^2.
	\end{equation*}
\end{asmp}
\noindent It is important to note that Assumption~\ref{assum:StrongConvexity} only requires that the {\emph{global}} loss function, $f_t(x)$, is strongly convex; the local loss $f_{i,t}$ needs not. In fact, each local loss function does not even need to be convex. Such cases could occur if the local objective function of one agent was strongly convex at each time and those of the remaining agents summed together to form a convex function, as in \cite{gade2016distributed}.
Note also that in existing literature in online optimization, e.g., \cite{mokhtari2016online}, the objective function is usually assumed to be Lipschitz or have uniformly bounded gradients, which usually requires the decision set to be compact. For unconstrained online optimization problems, this assumption is difficult to justify for some common objectives, e.g., quadratic functions. For the first time, we show that the Lipschitzness assumption holds for unconstrained online optimization problems at all the iterates of the online algorithm for common objective functions, as long as the time-varying objective functions satisfy Assumptions~\ref{assum:DynamicalMatrix}, \ref{assum:BoundedTimeVaryingFunction}, \ref{assum:LipschitzGradient} and \ref{assum:StrongConvexity}. 

In what follows, we assume that the agents tasked with solving problem~\eqref{eqn:ProblemForm} communicate subject to the graph $\mathcal{G} := (\mathcal{N}, \mathcal{E})$, where $\mathcal{N} = \{1, 2, \dots, N\}$ is the set of nodes indexed by the agents and $\mathcal{E}$ is the set of edges. If $(i, j) \in \mathcal{E}$, agent $i$ can receive information from its neighbor $j$. Moreover, we define by $W_{ij}$ the $i,j$-th entry of $W$ that captures the weight agent $i$ allocates to the information received from its neighbor $j$.  $W_{ij} = 0$ if $(i,j) \notin \mathcal{E}$. We make the following assumptions on the graph $\mathcal{G}$ and the weight matrix $W$.
\begin{asmp}
	\label{assum:DoublyStochasticity}
	The graph $\mathcal{G}$ is fixed, undirected, and connected, and the communication matrix $W$ is doubly stochastic. That is, $W \mathbf{1} = \mathbf{1}$ and $W^T \mathbf{1} = \mathbf{1}$.
\end{asmp}
The assumption that $W$ is doubly stochastic implies that $\|W - \frac{1}{n} \mathbf{1}\mathbf{1}^T\| = \sigma_W < 1$, where $\sigma_W$ is the mixing rate of the network. When $\sigma_W$ is smaller, the agents in the network reach consensus faster; see, e.g.,  \cite{zhang2018consensus}.

%% file: tex/Algorithm.tex
\begin{subequations}
	\label{eq:eq_oco_2}
	\begin{algorithm}[t]
		\caption{D-OCO with prediction} \label{alg:oco_nprediction}
			\KwIn{The primal variables $x_{i,0}$, the local gradients $\nabla f_{i,0}(x_{i,0})$ and global gradient estimates $y_{i,0} = \alpha \nabla f_{i,0}(x_{i,0})$ for all $i$. The estimated dynamic matrix $A$. $t = 0$.}{
			Agent $i$ computes \vspace{2mm}
				\begin{equation}
				\label{eq:Algorithm_step1}
				\hat{x}_{i,t+1} = \sum_{j \in N_i} W_{ij} (x_{j,t} - y_{j,t});
				\end{equation}
				
			Agent $i$ computes		
				\begin{equation}
				\label{eq:Algorithm_step2}
				x_{i,t+1} = A \hat{x}_{i, t+1};
				\end{equation}
				
			Agent $i$ computes \vspace{2mm}
				\begin{equation}
				\label{eq:Algorithm_step3}
				\begin{split}
				y_{i,t+1} = \sum_{j \in N_i} W_{ij} y_{j,t} & + \alpha \nabla f_{i, t+1}(x_{i,t+1})  \\
				& - \alpha \nabla f_{i, t}(x_{i,t});
				\end{split}
				\end{equation}
				
			$t \leftarrow t+1$, go to step 1.
		}
	\end{algorithm}
\end{subequations}

In this paper, we propose a new distributed gradient descent algorithm to solve the online optimization problem~\eqref{eqn:ProblemForm} that has an improved dynamic regret. 
Algorithm~\ref{alg:oco_nprediction} presents our proposed online distributed optimization algorithm with gradient tracking and prediction.

Specifically, every agent $i$ holds a local candidate optimal $x_{i,t}$ and a local gradient estimate $y_{i,t}$ based on its present local loss function. Using current information, each agent estimates its next local candidate optimal $\hat{x}_{i,t+1}$ using one step of gradient descent via

\begin{align}
	\hat{x}_{t+1} =( W \otimes I) (x_t - y_t), \nonumber
\end{align}
where $\hat{x}_t, x_t$ and $y_t$ stacks local variables $\hat{x}_{i,t}, x_{i,t}$ and $y_{i,t}$,
and then predict the next point using the given matrix $A$ of the dynamical system via

\begin{align}
	x_{t+1} = (I \otimes A) \hat{x}_{t+1}. \nonumber 
\end{align}
When the optimizer dynamics are not known, we can set $A$ to be the identity matrix. Then, the algorithm is reduced to the algorithm studied in \cite{hall2015online}.
After this, the next loss functions are revealed to each agent and the agents compute their individual $y_{i,t+1}$ by adding the $y_{i,t}$'s that have been communicated and combined with the matrix $W$ with a scaled gradient tracking step. The gradient tracking step is employed to correct for the change in objective function gradients \cite{qu2018harnessing,pu2020distributed} via

\begin{align}
	y_{t+1} = (W \otimes I) y_t + \alpha g_{t+1}(x_{t+1}) - \alpha g_{t}(x_{t}). \nonumber
\end{align}

%% file: tex/Convergence.tex
We now theoretically bound the dynamic regret of Algorithm~\ref{alg:oco_nprediction} under the assumptions~\ref{assum:DynamicalMatrix} - \ref{assum:DoublyStochasticity}. The outline of our analysis is as follows. First, Lemma~\ref{lemma:ConservationY} shows that the average of the local estimators $y_{i,t}$ can track the sum of the local gradients well. 
Then, we show that both the tracking error and the network error contract with some perturbation at each time step in Lemmas~\ref{lem:DynamicTrackingError} and \ref{lem:DynamicNetworkError}. The strong convexity assumption is necessary for proving the bound on the tracking error; the network error will be bounded via the gradient tracking employed in our algorithm. 
These bounds, together with Assumptions~\ref{assum:DynamicalMatrix} and \ref{assum:BoundedTimeVaryingFunction}, are then used to show that the iterates of Algorithm~\ref{alg:oco_nprediction} are confined within a bounded neighborhood around the changing optimizer in Lemma IV.5. This justifies the Lipschitzness of the objective function over this bounded neighborhood. The final regret bound is provided in Theorem~\ref{thm:BoundRegret}.

Define the gradient estimator $y_t = [y_{1,t}^T, y_{2,t}^T,$  $\dots, y_{N,t}^T]^T$. Then, we have the following conservation property of $y_t$.
\begin{lem}
	\label{lemma:ConservationY}
	Let Assumption~\ref{assum:DoublyStochasticity} hold, and assume also that the local estimator is initialized as $y_{i,0} = \alpha \nabla f_{i,0}(x_{i,0})$ for all $i$. Then, for all  $t$, we have that 
	
	\begin{equation*}
		(\mathbf{1}^T \otimes I) y_t = \alpha (\mathbf{1}^T \otimes I) g_t(x_t).
	\end{equation*}
\end{lem}
\begin{proof}
	We prove this statement by induction. By the initialization of $y_{i,0}$, it is easy to see that 
	
	\begin{equation*}
		(\mathbf{1}^T \otimes I) y_0 = \alpha (\mathbf{1}^T \otimes I) g_0(x_0).
	\end{equation*}
	Then, assuming that the lemma is true at time $t-1$, according to \eqref{eq:Algorithm_step3}, we have that
	
	\begin{equation*}
	\begin{split}
		(\mathbf{1}^T \otimes I) y_{t} = & \; (\mathbf{1}^T \otimes I) \big( (W \otimes I) \\
		&  \; y_{t-1} + \alpha g_t(x_t) - \alpha g_{t-1}(x_{t-1}) \big) \\
		= &\; \alpha (\mathbf{1}^T \otimes I) g_{t}(x_{t}),
	\end{split}
	\end{equation*}
	where the second equation is due to the induction assumption. This concludes the proof.
\end{proof}

Let $\bar{x}_t = \frac{1}{N} \sum_{i=1}^{N} x_{i,t}$. 
%
The next lemma characterizes the dynamics of the tracking error $\|\bar{x}_t - x_t^\ast\|$. 
\begin{lem}
	\label{lem:DynamicTrackingError}
	Let Assumptions~\ref{assum:DynamicalMatrix}, \ref{assum:LipschitzGradient},  \ref{assum:StrongConvexity},  and \ref{assum:DoublyStochasticity} hold. Then, when $\alpha \leq \frac{1}{L_g}$, the tracking error $\|\bar{x}_t - x_t^\ast\|$ satisfies the following inequality for all $t$,
	
	\begin{equation}
		\label{eqn:TrackingError}
		\begin{split}
		& \|\bar{x}_{t+1} - x_{t+1}^\star\|  \leq (1 - \frac{\alpha}{n} \mu) \|\bar{x}_t - x_t^\ast\| + \\ 
		& \frac{L_g}{\sqrt{n}}\alpha \|x_t - \mathbf{1} \otimes \bar{x}_t\| + \|Ax_t^\star - x_{t+1}^\star\|.
		\end{split}
	\end{equation}
\end{lem}
\begin{proof}
	First, using the definition of $\bar{x}_{t+1}$ and adding and subtracting $Ax_t^\ast$, we have that $\|\bar{x}_{t+1} - x_{t+1}^\star\| = \|\frac{1}{n}(\mathbf{1}^T \otimes I)x_{t+1} - Ax_t^\star + Ax_t^\star - x_{t+1}^\star\|$. Then, according to the updates in \eqref{eq:Algorithm_step1} and \eqref{eq:Algorithm_step2}, we can obtain that
	
	\begin{flalign*}
	\label{eqn:PfTracking1}
	& \|\bar{x}_{t+1} - x_{t+1}^\star\| \leq  \|\frac{1}{n}(\mathbf{1}^T \otimes I) (W \otimes A) (x_t -  y_t) \nonumber  \\
	& \quad \quad \quad \quad \quad \quad \quad - Ax_t^\star\|  \quad + \|Ax_t^\star - x_{t+1}^\star\|.  \nonumber
	\end{flalign*}
	According to Assumption~\ref{assum:DoublyStochasticity} and extracting matrix $A$, we get that $\|\frac{1}{n}(\mathbf{1}^T \otimes I) (W \otimes A) (x_t -  y_t) - Ax_t^\star\| = \| A(\bar{x}_t - \frac{1}{n}(\mathbf{1}^T\otimes I) y_t - x_t^\star)\|$. Therefore, we obtain that
	
	\begin{equation}
		\label{eqn:PfTracking1.1}
		\begin{split}
		\|\bar{x}_{t+1} - x_{t+1}^\star\|  \leq &  \| A\big(\bar{x}_t - \frac{1}{n}(\mathbf{1}^T\otimes I) y_t - x_t^\star\big) \| \\
		&  + \| Ax_t^\star - x_{t+1}^\star \|. 
		\end{split}
	\end{equation}
We now place an upper bound on the first term in the right hand side of Equation \eqref{eqn:PfTracking1.1}. Since $A$ has norm at most 1, it follows by the definition of the matrix norm that $\|Av \| \leq \| v \|.$ Thus, the first term on the right hand side of Equation \eqref{eqn:PfTracking1.1} is bounded above by $\|\bar{x}_t - \frac{1}{n}(\mathbf{1}^T\otimes I) y_t - x_t^\star\|$.
Recalling Lemma~\ref{lemma:ConservationY}, we can replace the term $(\mathbf{1}^T\otimes I) y_t$ in the above with $\alpha (\mathbf{1}^T \otimes I) g_{t}(x_{t})$. By doing this and using the triangle inequality, we can bound the quantity $\|\bar{x}_t - \frac{1}{n}(\mathbf{1}^T\otimes I) y_t - x_t^\star\|$ by a sum of the following two terms:

\begin{equation}
	\label{eqn:PfTrackign2.1}
	\|\bar{x}_t - \frac{\alpha}{n} (\mathbf{1}^T \otimes I) g_t(\mathbf{1} \otimes \bar{x}_t) - x_t^\ast\|
\end{equation}
\begin{equation}
	\label{eqn:PfTracking2.2}
	 \| \frac{\alpha}{n} (\mathbf{1}^T \otimes I) \big( g_t(\mathbf{1} \otimes \bar{x}_t) - g_t(x_t) \big) \|
\end{equation}
By using both Assumption~\ref{assum:StrongConvexity} and Lemma 10 in \cite{qu2018harnessing}, we have that

	\begin{flalign}
		\label{eqn:PfTracking3}
		& \|\bar{x}_t - \frac{1}{n}(\mathbf{1}^T \otimes I) y_t - x_t^\ast\|  \leq (1 - \frac{\alpha}{n} \mu) \|\bar{x}_t - x_t^\ast\| & \nonumber \\ 
		& \quad \quad \quad \quad \quad \quad \quad \quad \quad + \frac{L_g}{\sqrt{n}}\alpha \|x_t - \mathbf{1} \otimes \bar{x}_t\|. &
	\end{flalign}
	Combining inequality~\eqref{eqn:PfTracking1.1} and \eqref{eqn:PfTracking3}, we obtain the desired result in \eqref{eqn:TrackingError}.
\end{proof}
The next lemma characterizes the dynamics of the network errors  $\|x_t - \mathbf{1}\bar{x}_t\|$ and $\|y_t - \mathbf{1} \otimes \bar{y}_t\|$.
\begin{lem}
	\label{lem:DynamicNetworkError}
	Let Assumptions~\ref{assum:DynamicalMatrix}, \ref{assum:LipschitzGradient}, \ref{assum:StrongConvexity} and \ref{assum:DoublyStochasticity} hold. Then, the newtork errors $\|x_t - \mathbf{1}\bar{x}_t\|$ and $\|y_t - \mathbf{1} \otimes \bar{y}_t\|$ satisfy the following inequalities at all $t$, 
	
	\begin{equation}
	\label{eqn:NetworkError_x}
	\begin{split}
	\|x_{t+1} - \mathbf{1} \otimes \bar{x}_{t+1}\| \leq & \; \sigma_W \|x_t - \mathbf{1} \otimes \bar{x}_t\|  \\
	& + \sigma_W \|y_t - \mathbf{1} \otimes \bar{y}_t\|,
	\end{split}
	\end{equation}
	and 
	
	\begin{flalign}
	\label{eqn:NetworkError_y}
	& \|y_{t+1} - \mathbf{1} \otimes \bar{y}_{t+1}\| \leq \sigma_W(1 + L_g\alpha) \|y_t - \mathbf{1}\otimes\bar{y}_t\| & \nonumber \\
	&  + L_g(1 + \sigma_W + L_g\alpha) \alpha \|x_t - \mathbf{1} \otimes \bar{x}_t\| & \nonumber \\ 
	& + \sqrt{n}L_g(2 + L_g\alpha)\alpha\|\bar{x}_t - x_t^\ast\| & \\ 
	& + \alpha \|g_{t+1}(\mathbf{1} \otimes Ax_t^\ast) - g_t(\mathbf{1} \otimes x_t^\ast)\|. & \nonumber 
	\end{flalign}
\end{lem}
\begin{proof}
	First, consider the dynamics of $\|x_{t} - \mathbf{1} \otimes \bar{x}_t\|$. We have that 
$$\|x_{t+1} - \mathbf{1} \otimes \bar{x}_{t+1}\| = \|\big((I - \frac{1}{n}\mathbf{1}\mathbf{1}^T) \otimes I\big) x_{t+1}\|. $$ 
\noindent Recalling the updates in \eqref{eq:Algorithm_step1} and \eqref{eq:Algorithm_step2}, as well as Assumption~\ref{assum:DoublyStochasticity}
$$ \|x_{t+1} - \mathbf{1} \otimes \bar{x}_{t+1}\| =  \|\big( (W - \frac{1}{n}\mathbf{1}\mathbf{1}^T) \otimes A\big)  (x_t - y_t) \|.$$ 

\noindent Using the definition of the matrix norm and Assumption~\ref{assum:DynamicalMatrix}, we have that

	\begin{equation*}
	  \|x_{t+1} - \mathbf{1} \otimes \bar{x}_{t+1}\| \leq \|\big((W - \frac{1}{n}\mathbf{1}\mathbf{1}^T) \otimes I\big) (x_t - y_t)\|.
	\end{equation*}
\noindent By Assumption~\ref{assum:DoublyStochasticity}, we obtain that 

$$\big((W - \frac{1}{n}\mathbf{1}\mathbf{1}^T) \otimes I\big) (\mathbf{1} \otimes \bar{x}_t) = 0$$ 
\noindent and 
$$\big((W - \frac{1}{n}\mathbf{1}\mathbf{1}^T) \otimes I\big) (\mathbf{1} \otimes \bar{y}_t) = 0.$$ Therefore, we can add $\mathbf{1} \otimes \bar{y}_t - \mathbf{1} \otimes \bar{x}_t$ inside the term $(x_t - y_t)$ on the right hand side of the above inequality and obtain, using the triangle inequality, that

	\begin{equation}
		\begin{split}
		\|x_{t+1} - \mathbf{1} \otimes \bar{x}_{t+1}\| \leq & \; \sigma_W \|x_t - \mathbf{1} \otimes \bar{x}_t\| \\
		& + \sigma_W \|y_t - \mathbf{1} \otimes \bar{y}_t\|.
		\end{split}
	\end{equation}
	We now consider the dynamics of $\|y_t - \mathbf{1} \otimes \bar{y}_t\|$. We have that $\|y_{t+1} - \mathbf{1} \otimes \bar{y}_{t+1}\| = \|\big((I - \frac{1}{n}\mathbf{1} \mathbf{1}^T) \otimes I\big) y_{t+1}\|$. By \eqref{eq:Algorithm_step3}, we know $\|y_{t+1} - \mathbf{1} \otimes \bar{y}_{t+1}\|$ equals
	
$$\|\big((I - \frac{1}{n}\mathbf{1} \mathbf{1}^T) \otimes I\big) ((W \otimes I) y_t + \alpha g_{t+1}(x_{t+1}) - \alpha g_t(x_t))\|,$$

\noindent and another application of the triangle inequality above shows that $\|y_{t+1} - \mathbf{1} \otimes \bar{y}_{t+1}\|$ is bounded by the sum of the following two terms
\begin{equation}
\label{eqn:PfNetworking_1.1}
\|((W - \frac{1}{n}\mathbf{1} \mathbf{1}^T) \otimes I) y_t \|
\end{equation}
\begin{equation}
\label{eqn:PfNetworking_1.2}
\alpha \| ((I - \frac{1}{n}\mathbf{1} \mathbf{1}^T) \otimes I) (g_{t+1}(x_{t+1}) - g_t(x_t))\|
\end{equation}
Since $((W - \frac{1}{n}\mathbf{1} \mathbf{1}^T) \otimes I) (\mathbf{1} \otimes \bar{y}_t) = 0$ and, by Assumption~\ref{assum:DoublyStochasticity}, $\|W - \frac{1}{n}\mathbf{1} \mathbf{1}^T\| \leq \sigma_W$, we have 

$$\|((W - \frac{1}{n}\mathbf{1} \mathbf{1}^T) \otimes I) y_t \| \leq \sigma_W \|y_t - \mathbf{1} \otimes \bar{y}_t\|.$$ 

Furthermore, since $\|I - \frac{1}{n}\mathbf{1} \mathbf{1}^T\| \leq 1$, combining the above discussion with \eqref{eqn:PfNetworking_1.1} and \eqref{eqn:PfNetworking_1.2}, we get that

	\begin{equation}
		\label{eqn:PfNetworking_2}
		\begin{split}
		\|y_{t+1} - \mathbf{1} \otimes \bar{y}_{t+1}\|  & \leq \sigma_W \|y_t - \mathbf{1}\otimes \bar{y_t}\| \\ 
& + \alpha\|g_{t+1}(x_{t+1}) - g_t(x_t)\|. \\
		\end{split}
	\end{equation}
	Adding and subtracting the terms $g_{t+1}(\mathbf{1} \otimes Ax_t^\ast)$ and $g_t(\mathbf{1} \otimes x_t^\ast)$ inside the norm $\|g_{t+1}(x_{t+1}) - g_t(x_t)\|$, and using the triangle inequality, we obtain that
	
	\begin{flalign}
		\label{eqn:PfNetworking_5}
		& \quad \|g_{t+1}(x_{t+1}) - g_t(x_t)\| & \nonumber \\
		& \leq \|g_{t+1}(x_{t+1}) - g_{t+1}(\mathbf{1} \otimes Ax_t^\ast) \| \nonumber \\
		&+ \| g_t(\mathbf{1} \otimes x_t^\ast) - g_t(x_t) \| \nonumber \\
		&+ \|  g_{t+1}(\mathbf{1} \otimes Ax_t^\ast) - g_t(\mathbf{1} \otimes x_t^\ast) \| & \nonumber \\
		& \leq  L_g  \|x_{t+1} - \mathbf{1} \otimes Ax_t^\ast\| +  L_g  \|x_{t} - \mathbf{1} \otimes x_t^\ast\| & \nonumber \\
		& \quad + \|g_{t+1}(\mathbf{1} \otimes Ax_t^\ast) - g_t(\mathbf{1} \otimes x_t^\ast)\|,
	\end{flalign}
	where the second inequality is due to Assumption~\ref{assum:LipschitzGradient}. Using the updates of $x_{t+1}$ in \eqref{eq:Algorithm_step1} and \eqref{eq:Algorithm_step2}, we have that 

\begin{align*}
\|x_{t+1} - \mathbf{1} \otimes Ax_t^\ast\| &= \|(I \otimes A) \hat{x}_{t+1} - \mathbf{1} \otimes Ax_t^\ast\| \\
&\leq \|A\| \|\hat{x}_{t+1} - \mathbf{1} \otimes x_t^\ast\| \\
&\leq \|\hat{x}_{t+1} - \mathbf{1} \otimes x_t^\ast\|
\end{align*} 

\noindent where the last two inequalies are due to the Cauchy-Schwartz inequality and Assumption~\ref{assum:DynamicalMatrix}. Replacing $\hat{x}_{t+1}$ with $(W \otimes I) (x_t - y_t)$ according to \eqref{eq:Algorithm_step1}, we have that

	\begin{equation*}
		\|x_{t+1} - \mathbf{1} \otimes Ax_t^\ast\|  \leq \| (W \otimes I) (x_t - y_t) - \mathbf{1} \otimes x_t^\ast \|.
	\end{equation*} 
	Adding and subtracting the terms $\mathbf{1} \otimes \bar{x}_t$ and $\mathbf{1} \otimes \bar{y}_t$ inside the norm on the right hand side of the above inequality, and using the triangle inequality, we get that
	
	\begin{equation}
		\label{eqn:PfNetworking_3}
		\begin{split}
		\|x_{t+1} - \mathbf{1} \otimes Ax_t^\ast\|  &\leq \| (W \otimes I) x_t - \mathbf{1} \otimes \bar{x}_t\| \\
		&+ \| (W \otimes I) y_t -  \mathbf{1} \otimes \bar{y}_t \| \\
		&+ \| \mathbf{1} \otimes \bar{x}_t - \mathbf{1} \otimes x_t^\ast \| + \| \mathbf{1} \otimes \bar{y}_t \|.
		\end{split}
	\end{equation}
	Next, we provide upper bounds on the terms on the right hand side of \eqref{eqn:PfNetworking_3} respectively. By Assumption~\ref{assum:DoublyStochasticity}, we have 
	
$$(W \otimes I) x_t - \mathbf{1} \otimes \bar{x}_t =  (W \otimes I) (x_t - \mathbf{1} \otimes \bar{x}_t ).$$ 
\noindent Moreover, since $(\frac{1}{n} \mathbf{1} \mathbf{1}^T \otimes I) (x_t - \mathbf{1} \otimes \bar{x}_t) = 0$, we have that 

$$(W \otimes I) x_t - \mathbf{1} \otimes \bar{x}_t = ( (W - \frac{1}{n} \mathbf{1} \mathbf{1}^T)  \otimes I) (x_t - \mathbf{1} \otimes \bar{x}_t ).$$ By Assumption~\ref{assum:DoublyStochasticity} and the Cauchy-Schwartz inequality, we have that 

	\begin{subequations}
		\begin{equation}
			\label{eqn:PfNetworking_3.1}
			\| (W \otimes I) x_t - \mathbf{1} \otimes \bar{x}_t) \| \leq \sigma_W \| x_t - \mathbf{1} \otimes \bar{x}_t \|.
		\end{equation}
	 Similarly, we obtain that
	 
	 	\begin{equation}
	 		\label{eqn:PfNetworking_3.2}
	 		\| (W \otimes I) y_t - \mathbf{1} \otimes \bar{y}_t) \| \leq \sigma_W \| y_t - \mathbf{1} \otimes \bar{y}_t \|.
	 	\end{equation}
	Furthermore, we have that
	
	\begin{equation}
		\label{eqn:PfNetworking_3.3}
		\|\mathbf{1} \otimes \bar{x}_t - \mathbf{1} \otimes x_t^\ast\| = \sqrt{n} \|\bar{x}_t - x_t^\ast \|.
	\end{equation}
In addition, we have that 

$$\|\mathbf{1} \otimes \bar{y}_t\| = \| \mathbf{1} \otimes \bar{y}_t -  \alpha(\frac{1}{n}\mathbf{1}\mathbf{1}^T \otimes I) g_t(\mathbf{1}\otimes x_t^\ast) \|$$ 
\noindent because of the definition of $x_t^\ast$ and the fact that $(\frac{1}{n}\mathbf{1}\mathbf{1}^T \otimes I) g_t(\mathbf{1}\otimes x_t^\ast) = 0$. Recalling Lemma~\ref{lemma:ConservationY}, we have that $\mathbf{1} \otimes \bar{y}_t = \alpha (\frac{1}{n} \mathbf{1}\mathbf{1}^T \otimes I) g_t(x_t)$. Therefore, we have that 

$$\|\mathbf{1} \otimes \bar{y}_t\| = \|\alpha (\frac{1}{n} \mathbf{1}\mathbf{1}^T \otimes I) (g_t(x_t) - g_t(\mathbf{1}\otimes x_t^\ast))\|.$$ Using the Cauchy-Schwartz inequality and the fact that $\|\frac{1}{n} \mathbf{1}\mathbf{1}^T\| = 1$, we see that

$$\|\mathbf{1} \otimes \bar{y}_t\|  \leq \alpha \| g_t(x_t) - g_t(\mathbf{1}\otimes x_t^\ast) \|.$$ Then, according to Assumption~\ref{assum:LipschitzGradient}, we get that

	\begin{equation}
		\label{eqn:PfNetworking_3.4}
		\|\mathbf{1} \otimes \bar{y}_t\| \leq  \alpha L_g \|x_t - \mathbf{1} \otimes x_t^\ast \|.
	\end{equation}
\end{subequations}
	Combining the bounds in \eqref{eqn:PfNetworking_3.1}-\eqref{eqn:PfNetworking_3.4} with inequality~\eqref{eqn:PfNetworking_3}, we obtain that	
	\begin{flalign}
		\label{eqn:PfNetworking_4}
		\|x_{t+1} - \mathbf{1} \otimes Ax_t^\ast\|  &\leq \sigma_W \| x_t - \mathbf{1} \otimes \bar{x}_t \| \nonumber \\
		&+ \sigma_W \| y_t -\mathbf{1} \otimes \bar{y}_t \| \nonumber \\
		&+  \sqrt{n} \|\bar{x}_t - x_t^\ast \| \nonumber \\
		&+ \alpha L_g \|x_t - \mathbf{1} \otimes x_t^\ast \|.
	\end{flalign}
	Combining inequality~\eqref{eqn:PfNetworking_4} with \eqref{eqn:PfNetworking_5}, we have that
	
	\begin{flalign}
& \|g_{t+1}(x_{t+1}) - g_t(x_t)\| \leq \sigma_W L_g  \| x_t - \mathbf{1} \otimes \bar{x}_t \|  \nonumber \\
&+ \sigma_W L_g  \nonumber \| y_t -\mathbf{1}  \otimes \bar{y}_t \| + \sqrt{n} L_g  \|\bar{x}_t - x_t^\ast \| \nonumber \\
&+ L_g(1 + \alpha L_g) \| x_t - \mathbf{1} \otimes x_t^\ast \| \nonumber \\
&+  \|g_{t+1}(\mathbf{1} \otimes Ax_t^\ast) - g_t(\mathbf{1} \otimes x_t^\ast)\|. & \nonumber 
\end{flalign}
Adding and subtracting $\mathbf{1} \otimes \bar{x}_t$ in $\|  x_t - \mathbf{1} \otimes x_t^\ast \|$ on the right hand side of the above inequality, using the triangle inequality, and rearranging terms, we get

\begin{flalign}
\label{eqn:PfNetworking_6}
& \|g_{t+1}(x_{t+1}) - g_t(x_t)\| \leq \sigma_W L_g  \| y_t -\mathbf{1}  \otimes \bar{y}_t \| \nonumber \\
&+ L_g (1 + \sigma_W + \alpha L_g) \| x_t - \mathbf{1}  \otimes \bar{x}_t \| \nonumber \\
&+ \sigma_W L_g  \| y_t -\mathbf{1}  \otimes \bar{y}_t \| + \sqrt{n} L_g (2 + \alpha L_g) \|\bar{x}_t  - x_t^\ast \| \nonumber \\ 
&+  \|g_{t+1}(\mathbf{1} \otimes Ax_t^\ast) - g_t(\mathbf{1} \otimes x_t^\ast)\|. 
\end{flalign}
Combining inequality~\eqref{eqn:PfNetworking_6} with \eqref{eqn:PfNetworking_2}, we obtain the desired result in \eqref{eqn:NetworkError_y}.
\end{proof}

Using the above lemmas, we now characterize the tracking performance of Algorithm~\ref{alg:oco_nprediction}. A byproduct of the analysis that follows is that the iterates of Algorithm~\ref{alg:oco_nprediction} are confined within a bounded neighborhood of the changing optimizers for all $t$. Specifically, let $z_t = [\|\bar{x}_{t} - x_t^\ast\|, \|x_t - \mathbf{1}\otimes\bar{x}_t\|, \|y_t - \mathbf{1}\otimes\bar{y}_t\|]^T$. Then, we can show the following result.
%

\begin{thm}
	\label{lem:BoundIteration}
	Let Assumptions~\ref{assum:DynamicalMatrix}, \ref{assum:BoundedTimeVaryingFunction}, \ref{assum:LipschitzGradient}, \ref{assum:StrongConvexity},  and \ref{assum:DoublyStochasticity} hold. Moreover, assume that the stepsize $\alpha$ satisfies
	
	\begin{align}\label{eqn:StepsizeCondition}
	\alpha \leq \min\bigl\{ \frac{n}{\mu}, & \; \frac{(1 - \sigma_W) (1 - \sqrt{\sigma_W})}{\sigma_W (2 + \sigma_W) + 3\sigma_W \frac{n}{\mu} L_g} \frac{1}{L_g}, \nonumber \\
	& \frac{1}{L_g}, \frac{1-\sqrt{\sigma_W}}{\sqrt{\sigma_W}} \frac{1}{L_g}\bigr\}.
	\end{align}
	Then, by running Algorithm~\ref{alg:oco_nprediction}, we have that
	
	\begin{align}
		\limsup_{t \rightarrow \infty} \|x_{i,t} - x_t^\ast\| \leq \frac{2}{C_{inv}} \max\{c_{ij}\}(C_w + C_g), \nonumber
	\end{align}
	where the expressions of $C_{inv}$ and $c_{ij}$'s are provided in Lemma A.1.
	Therefore, there exists a constant $C_z$, such that $\| x_{i,t} - x_t^\ast \| \leq C_z$ for all $i$ and $t$. Furthermore, the objective functions $f_{i,t}(x)$ have bounded gradients over the $C_z-$neighborhood around the optimizer $x_t^\ast$ for all $i$ and $t$, i.e., $\| \nabla f_{i,t}(x_{i,t}) \| \leq G$ for all $\| x_{i,t} - x_t^\ast \| \leq C_z$.
\end{thm}

\begin{proof}
	According to Lemmas~\ref{lem:DynamicTrackingError} and \ref{lem:DynamicNetworkError}, and for $\alpha \leq \min\{\frac{1}{L_g}, \frac{1-\sqrt{\sigma_W}}{\sqrt{\sigma_W}} \frac{1}{L_g} \}$ we have that
	\begin{align}\label{eqn:BoundIteration_1}
		z_{t+1} \preceq \Phi(\alpha) z_t + d_t,
	\end{align}
	where $\preceq$ denotes element-wise inequality, and the matrix $\Phi(\alpha)$ and vector $d_t$ are defined as
	
	\begin{align}
		\Phi(\alpha) = \begin{bmatrix}
		1 - \frac{\mu}{n}\alpha & \frac{L_g}{\sqrt{n}}\alpha & 0 \\
		0 & \sigma_W & \sigma_W \\
		3\sqrt{n}L_g\alpha & L_g(2 + \sigma_W)\alpha & \sqrt{\sigma_W}\\
		\end{bmatrix}, \text{ and }\nonumber
	\end{align}
	
	\begin{align}
		d_t = [ \|Ax_t^\ast - x_{t+1}^\ast\|, 0, \alpha \|g_{t+1}(\mathbf{1}\otimes Ax_t^\ast) - g_{t}(\mathbf{1}\otimes x_t^\ast)\| ]^T. \nonumber
	\end{align}
	Using inequality~\eqref{eqn:BoundIteration_1}, we have that $z_{1} \preceq \Phi(\alpha) z_0 + d_0$, and, therefore, $z_{2} \preceq \Phi(\alpha) z_1 + d_1 \preceq \Phi(\alpha)^2 z_0 + \Phi(\alpha) d_0 + d_1$. Repeating this process until time $t$, we obtain that
	
	\begin{align}\label{eqn:BoundIteration_2}
		z_t \preceq \Phi(\alpha)^t z_0 + \sum_{j = 0}^{t - 1} \Phi(\alpha)^{t-1-j} d_j.
	\end{align}
	Denote $\vec{d} = [C_w, 0, C_g]^T$. Since $\|Ax_t^\ast - x_{t+1}^\ast\| \leq C_w$ and $\|g_{t+1}(\mathbf{1}\otimes Ax_t^\ast) - g_{t}(\mathbf{1}\otimes x_t^\ast)\| \leq C_g$ for all $t$,  we have that all $d_j$'s in \eqref{eqn:BoundIteration_2} can be bounded as $d_j \preceq \vec{d}$. Therefore, we get that $\sum_{j = 0}^{t - 1} \Phi(\alpha)^{t-1-j} d_j \preceq \big( \sum_{j = 0}^{t - 1} \Phi(\alpha)^{t-1-j} \big) \vec{d} \preceq  \big( \sum_{t = 0}^{\infty} \Phi(\alpha)^{t} \big) \vec{d}$. According to \eqref{eqn:BoundIteration_2}, we have that 
	
	\begin{align}
		z_t \preceq \Phi(\alpha)^t z_0 + \big( \sum_{t = 0}^{\infty} \Phi(\alpha)^{t} \big) \vec{d} = \Phi(\alpha)^t z_0 + \big( I - \Phi(\alpha)\big)^{-1} \vec{d}, 	\nonumber
	\end{align}
	where the equality above is due to the fact $\big( I - \Phi(\alpha)\big)^{-1}  = \sum_{t = 0}^{\infty} \Phi(\alpha)^{t}$. From Lemma A.1,  we have that the spectral radius $\rho(\Phi(\alpha)) < 1$ when the stepsize $\alpha$ is selected as in~\eqref{eqn:StepsizeCondition}. Then, all the entries in the matrix $\Phi(\alpha)^t$ go to $0$ when $t$ goes to infinity.  Therefore, we have that
	
	\begin{align}
		\limsup_{t \rightarrow \infty} z_t \preceq \big( I - \Phi(\alpha)\big)^{-1} \vec{d}. \nonumber
	\end{align}
	From Lemma A.1, all the entries in $\big( I - \Phi(\alpha)\big)^{-1}$ are smaller than $\frac{1}{C_{inv}}\max\{c_{ij}\}$. Therefore, we have that each entry in $\limsup_{t \rightarrow \infty} z_t$ is smaller than $\frac{1}{C_{inv}} \max\{c_{ij}\}(C_w + C_g)$.
	Recalling the definition of $z_t$, since $\|x_{i,t} - x_t^\ast\| \leq \|x_{i,t} - \bar{x}_t\| + \|\bar{x}_t - x_t^\ast\|$, we have that 
	
	\begin{align}
		\limsup_{t \rightarrow \infty} \|x_{i,t} - x_t^\ast\| \leq \frac{2}{C_{inv}} \max\{c_{ij}\}(C_w + C_g). \nonumber
	\end{align}
	Therefore, it is straightforward to see that $\| x_{i,t} - x_t^\ast \|$ is also uniformly bounded for all $i$ and $t$. 
	By Assumption~\ref{assum:LipschitzGradient}, we have that $\|\nabla f_{i,t}(x_{i,t})  - \nabla f_{i,t}(x_t^\ast) \| \leq L_g \| x_{i,t} - x_t^\ast \| $. Furthermore, since  $\|\nabla f_{i,t}(x_{i,t})  - \nabla f_{i,t}(x_t^\ast) \| \geq \|\nabla f_{i,t}(x_{i,t})\| - \| \nabla f_{i,t}(x_{i,t}) \|$, we get that $\|\nabla f_{i,t}(x_{i,t})\| \leq \|\nabla f_{i,t}(x_t^\ast) \| +  L_g \| x_{i,t} - x_t^\ast \|$. Because both $\|\nabla f_{i,t}(x_t^\ast) \|$ and $\|\nabla f_{i,t}(x_t^\ast) \|$ are uniformly bounded, $\|\nabla f_{i,t}(x_{i,t})\|$ is also uniformly bounded. The proof is complete.
\end{proof}

Next, we present the theoretical upper bound on the dynamic regret of Algorithm~\ref{alg:oco_nprediction}.

\begin{thm}
	\label{thm:BoundRegret}
	Let Assumptions~\ref{assum:DynamicalMatrix}, \ref{assum:BoundedTimeVaryingFunction}, \ref{assum:LipschitzGradient},  \ref{assum:StrongConvexity},  and \ref{assum:DoublyStochasticity} hold. Moreover, select the stepsize $\alpha$ as in~\eqref{eqn:StepsizeCondition}. 
	Then, we have that the regret $R_T^d \leq K(C_1 + C_2 + C_3 + \mathcal{P}_T^A + \mathcal{V}_T^A)$, where $C_1 = \|\bar{x}_0 - x_0^\star\|$, $C_2 = \|x_0 - \mathbf{1} \otimes \bar{x}_0\|$, $C_3 = \|y_0 - \mathbf{1} \otimes \bar{y}_0\|$ and $K = 2n \frac{G}{C_{inv}} \max\{c_{ij}\}$, where the expressions of $C_{inv}$ and $c_{ij}$'s are provided in Lemma A.1.
\end{thm}
\begin{proof}
	Using Lemma~\ref{lem:BoundIteration} and following the steps in the proof of Lemma III.1 in \cite{ZhangCDC2019_OnlineOptimization}, we have that the regret $R_T^d$ is upper bounded by
	
	\begin{equation}\label{eqn:RegretBound_1}
		\begin{matrix}
		R_T^d \leq & \hspace{-2mm} n G \sum_{t=0}^{T} \|\bar{x}_t - x_t^\star\| & \hspace{-2mm} + & \hspace{-3mm} \sqrt{n} G \sum_{t=0}^T \|x_t - \mathbf{1} \otimes \bar{x}_t\|. 
		\end{matrix}
	\end{equation}
	Recall also the definition of the variable $z_t = [\|\bar{x}_{t} - x_t^\ast\|, \|x_t - \mathbf{1}\otimes\bar{x}_t\|, \|y_t - \mathbf{1}\otimes\bar{y}_t\|]^T$. Our goal is  to bound $\sum_{t = 0}^T z_t$. To do so, we sum up the inequality \eqref{eqn:BoundIteration_2} over $t$ to get
	
	\begin{align}
		\sum_{t = 0}^T z_t & \preceq \sum_{t = 0}^T \Phi(\alpha)^t z_0 + \sum_{t = 1}^T \sum_{j = 0}^{t - 1} \Phi(\alpha)^{t-1-j} d_j \nonumber \\
		& =\sum_{t = 0}^T \Phi(\alpha)^t z_0 + \sum_{j = 0}^{T-1} \sum_{t = 0}^{T - 1 - j} \Phi(\alpha)^{t} d_j \nonumber \\
		& \preceq \sum_{t = 0}^\infty \Phi(\alpha)^t z_0 + \sum_{t = 0}^{\infty} \Phi(\alpha)^{t}  \sum_{j = 0}^{T-1} d_j,	
	\end{align}
	where the second equality follows from rearranging terms and the third inequality is
	due to the fact that every entry in the matrix $\sum_{t = 0}^J \Phi(\alpha)^t$ is smaller than the corresponding entry in the matrix $\sum_{t = 0}^\infty \Phi(\alpha)^t$ for a finite number $J$. Since $\sum_{t = 0}^{\infty} \Phi(\alpha)^{t} = \big( I - \Phi(\alpha)\big)^{-1} $, using Lemma A.1, we have that every entry in the vector $\sum_{t = 0}^T z_t$ is bounded by
	
	\begin{align}\label{eqn:RegretBound_2}
		\frac{1}{C_{inv}} \max\{ c_{ij} \} \big( C_1 + C_2 + C_3 + \mathcal{P}_T^A + \mathcal{V}_T^A \big),
	\end{align}
	where the expressions of $C_{inv}$ and $c_{ij}$'s are provided in Lemma A.1. 
	Using the definition of the vector $z_t$, we can combine the bound on the entries of $z_t$ in \eqref{eqn:RegretBound_2} with the inequality~\eqref{eqn:RegretBound_1}. The proof is complete.
\end{proof}

\begin{rem}
	The upper bound on the stepsize $\alpha$ in \eqref{eqn:StepsizeCondition} depends on the global network parameter $\sigma_W$, which is usually unknown at each local agent. A similar dependency also appears in other distributed online opitmization methods,\cite{zhao2019decentralized,shahrampour2018distributed,nazari2019dadam}. In practice, since $\sigma_W$ only affects the upper bound on the stepsize, we can always select a stepsize small enough to satisfy \eqref{eqn:StepsizeCondition}. Alternatively, we can compute the upper bound in \eqref{eqn:StepsizeCondition}, for a fixed communication matrix $W$ by letting  the agents estimate $W$ through multiple rounds of communication before running the algorithm.
\end{rem}

%% file: tex/Experiements.tex
In this section, we investigate the performance of our proposed algorithm with numerical experiments based on a target tracking example. Specifically, we consider a sensor network consisting of $n=6$ nodes that collaboratively track $3$ time-varying signals of sinusoidal shape. Each signal $x_{j,t}^\ast$ can be written following form

\begin{equation*}
	\label{eqn:SinusoidalSignal}
	x_{j,t}^\ast = \begin{bmatrix}
	p_{j,t} \\ \dot{p}_{j,t}
	\end{bmatrix} = 
	\begin{bmatrix}
	A_j \sin(\omega_{j} t + \phi_j) \\ \omega_j A_j \cos (\omega_{j}t + \phi_j)
	\end{bmatrix},
\end{equation*}
where $p_{j,t}$ is the position, $\dot{p}_{j,t}$ is the velocity of target $j$, $A_j$ is the amplitude, $\omega_j$ is the angular frequency, and $\phi_j$ is the phase of the signal. Each signal is subject to an ODE with noise

\begin{equation}
	\label{eqn:SinusoidalODE}
	\dot{x}_{j,t}^\ast = \begin{bmatrix}
	0 & 1 \\ -\omega_{j}^2 & 0
	\end{bmatrix} x_{j,t}^\ast + w_{j,t},
\end{equation}
where $w_j$ is a zero mean Gaussian noise. The matrix governing the dynamical system in this example can be estimated by discretizing the solution of the ODE \eqref{eqn:SinusoidalODE}. In the simulation, the amplitudes $\{A_j\}$ and inital phases are uniformly generated from the intervals $[0, 2]$ and $[0, \pi]$, respectively. The doubly stochastic matrix $W$ used in step \eqref{eq:Algorithm_step1} is randomly generated and $\sigma_W = 0.8554$. The sampling frequency is set to $100$Hz. The measurement model of sensor $i$ at time $t$ is

\begin{equation}
	\label{eqn:Measurement}
	y_{i,t} = C_{i} x_{t}^\ast,
\end{equation}
where $x_t^\ast = [\dots; x_{j,t}^\ast; \dots]$ is the true target state, $y_{i,t} \in \mathbb{R} $ is the observation at sensor $i$, and $C_i \in \mathbb{R}^{1 \times 6}$ is the measurement matrix that is randomly generated. At time $t$, the global problem is defined as

\begin{equation}
	\label{eqn:GlobalSensing}
	\min_{x_t} \frac{1}{2} \sum_{i = 1}^{6} |C_{i}x_t - y_{i,t}|^2 \triangleq \frac{1}{2} \|Cx_t - y_t\|^2,
\end{equation}
where the matrix $C$ and the vector $y_t$ stack all measurement matrices $C_i$ and local measurements $y_{i,t}$. We assume that $C^{T}C$ is positive definite in order to guarantee that the objective is strongly convex. This guarantees that the target state $x_t$ is determined uniquely given all local observations at time $t$. 
\begin{figure}
	\centering
	\subfigure[Periodicity $1000$s]{\includegraphics[width=.8\columnwidth]{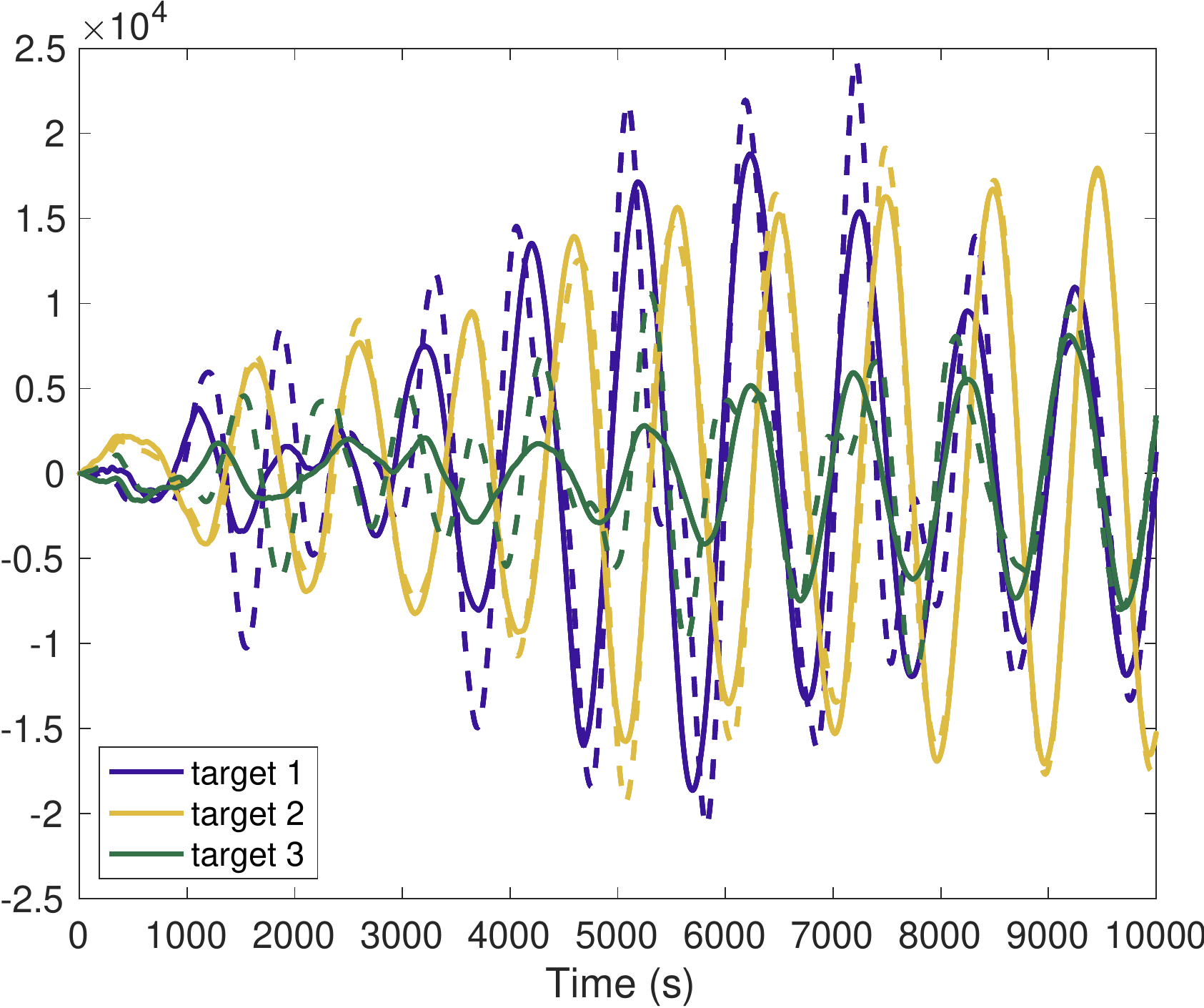}}
	\subfigure[Periodicity $10$s]{\includegraphics[width=.8\columnwidth]{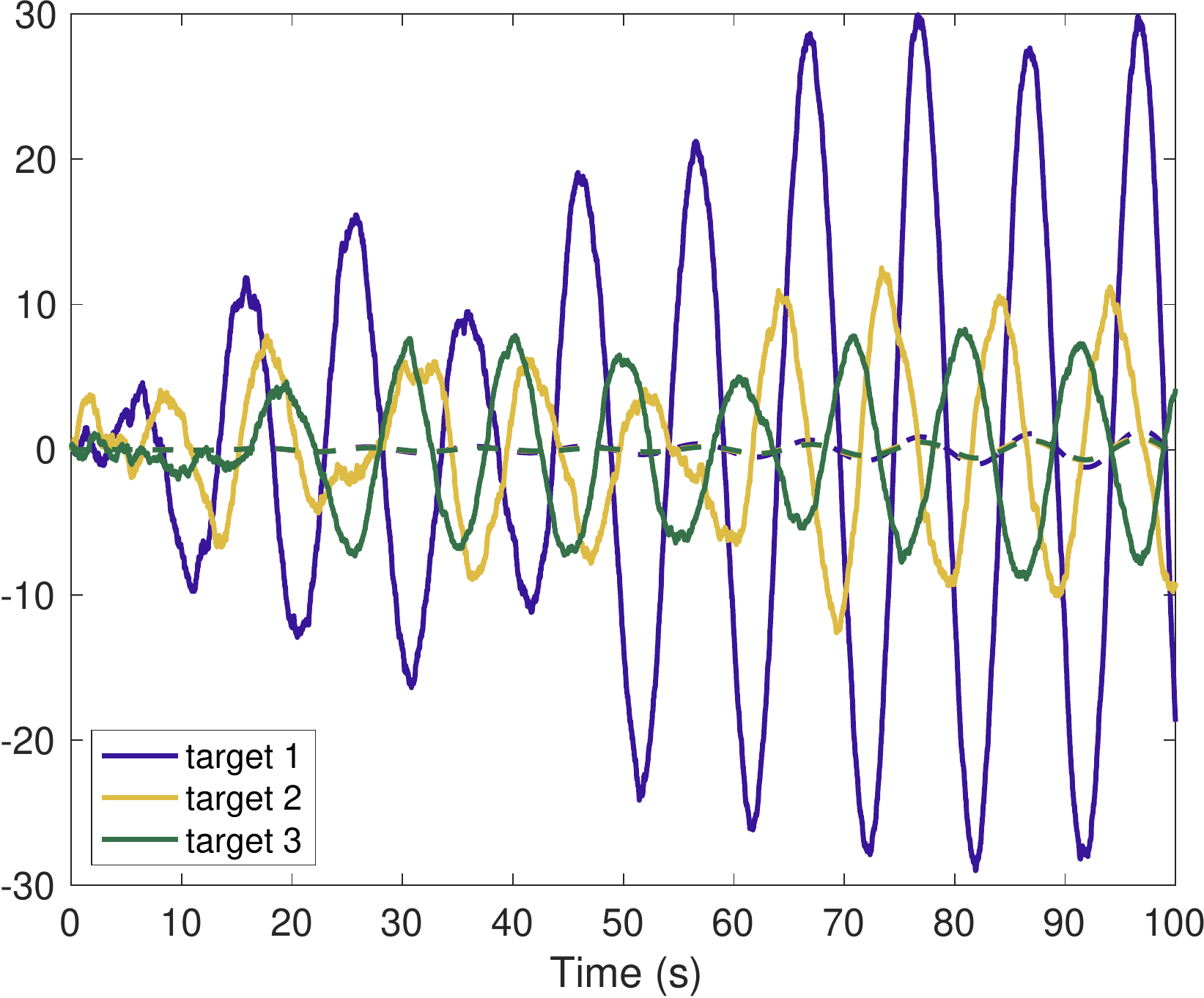}}	
	\caption{Tracking performance of Algorithm~\ref{alg:oco_nprediction} with stepsize being selected according to the theoretical bound in Theorem~\ref{thm:BoundRegret}. In both figures, the dashed colored curves represent the true positions of the three targets. The solid curves represent the position estimates of the three targets at local agent $1$.}. 
	\label{fig:tracking_theoreticalStepsize}
\end{figure}

\subsection{Choice of stepsize}
In this section, we track signals defined as in \eqref{eqn:SinusoidalODE} with periodicity $1000$s and $10$s using Algorithm~\ref{alg:oco_nprediction} and assuming that $A$ is the identity matrix, i.e. we do not take any predictability of the future into account. In both experiments, we let the step size $\alpha = 2.45 \times 10^{-5}$ in accordance with the assumptions of Theorem~\ref{thm:BoundRegret}. The tracking performance is shown in Figure~\ref{fig:tracking_theoreticalStepsize}. We observe that using this theoretical stepsize, Algorithm~\ref{alg:oco_nprediction} performs better in tracking signals with periodicity $1000$s than $10$s. 
This is unsurprising; the motion of the targets is sampled more frequently when the periodicity is 1000s.
In particular, the theoretical step size $\alpha$ is too conservative, preventing our algorithm from tracking quickly moving targets. However, we can use larger stepsizes than the theoretical ones given in Theorem~\ref{thm:BoundRegret} and achieve better performance. Specifically, in Figure~\ref{fig:tracking_EmpStep} (a), we observe that for a more aggressive selection of the stepsize, the accumulated regret becomes smaller, implying better tracking performance. And in Figure~\ref{fig:tracking_EmpStep} (b), we show that the tracking performance of Algorithm~\ref{alg:oco_nprediction} for the target signal of periodicity $10$s can be improved using $\alpha = \frac{1}{2L_g} \approx 0.2$.
\begin{figure}
	\centering
	\subfigure[Regret under different stepsizes]{\includegraphics[width=.8\columnwidth]{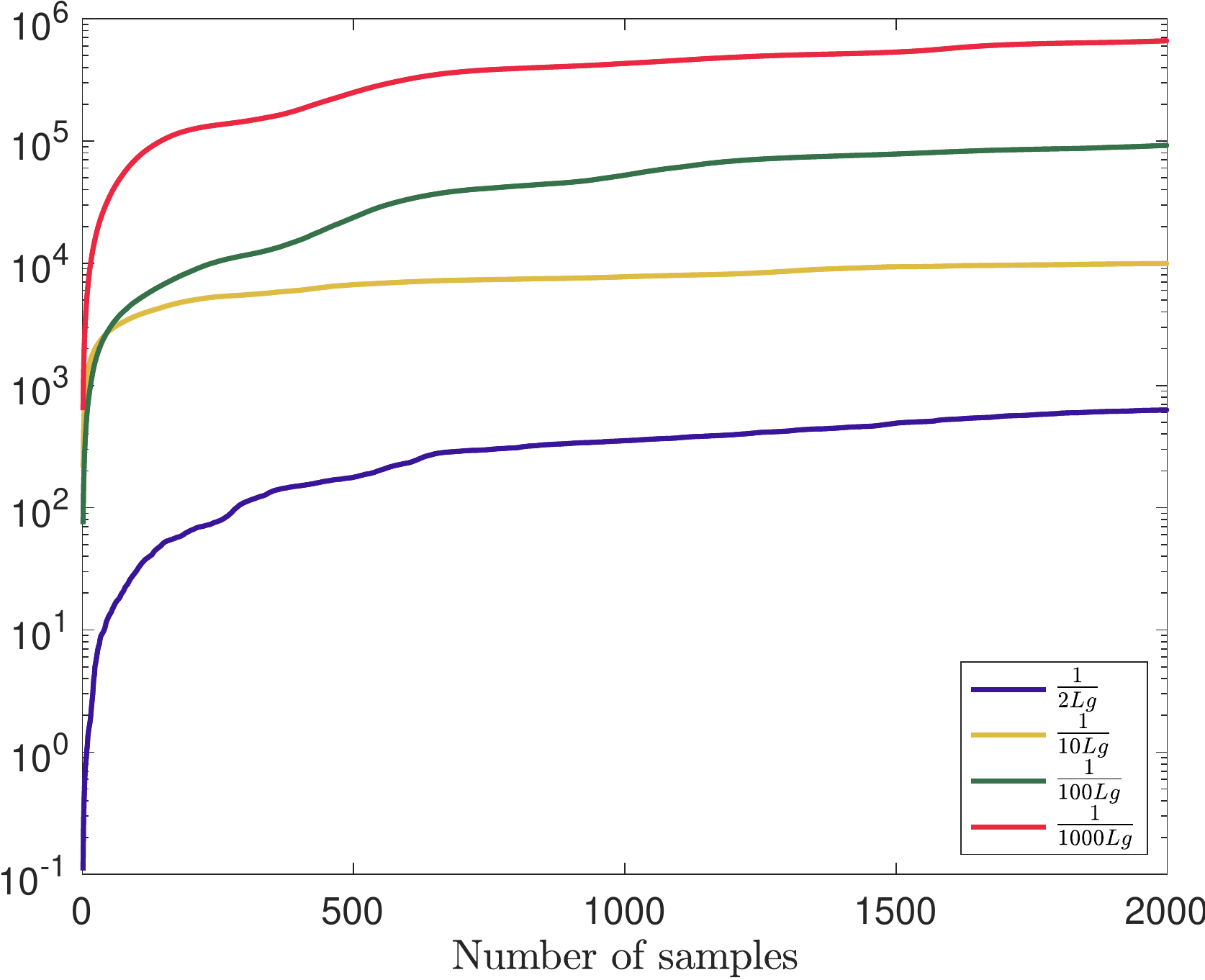}}
	\subfigure[Tracking performance when $\alpha = \frac{1}{2L_g}$]{\includegraphics[width=.8\columnwidth]{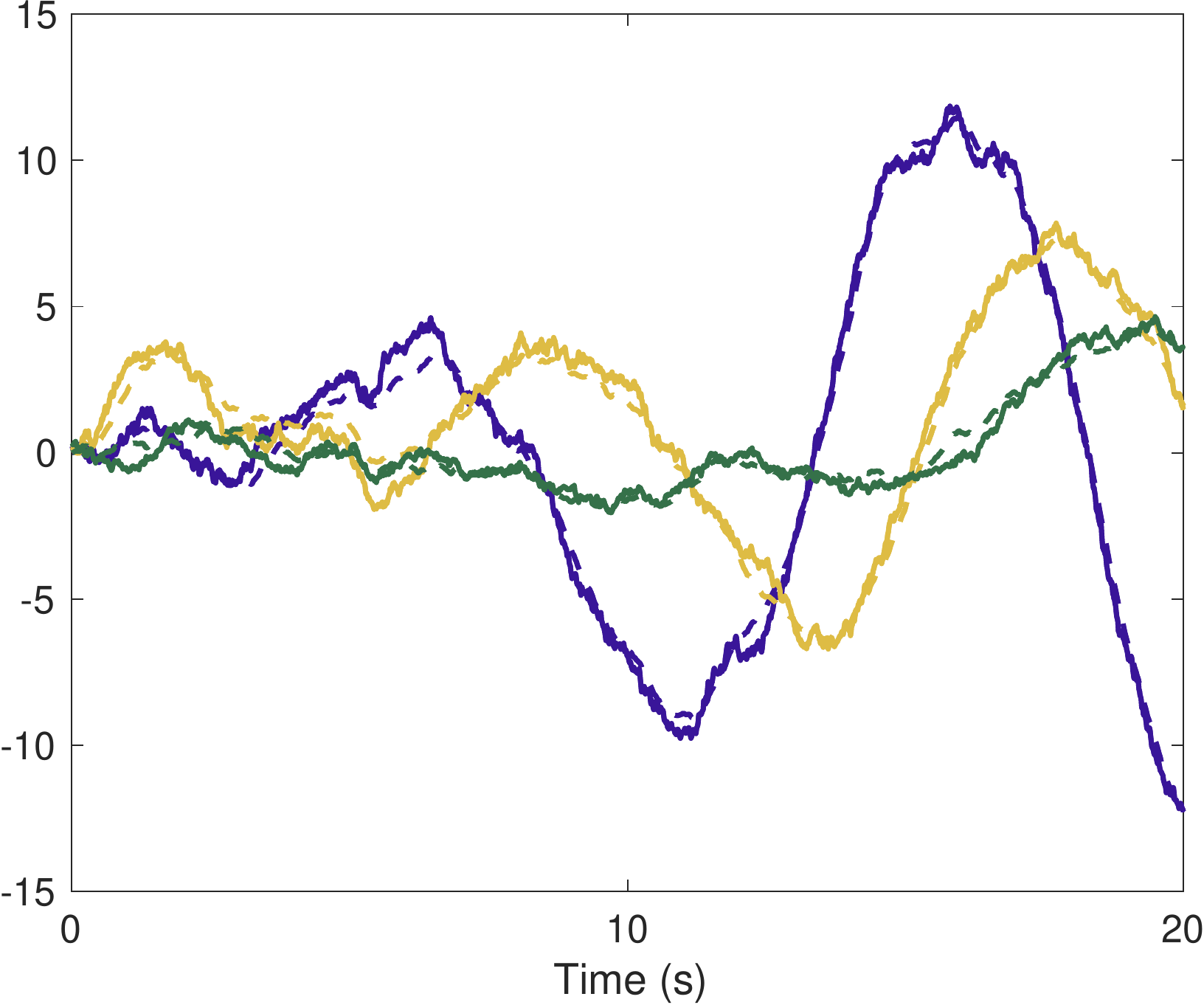}}	
	\caption{Tracking performance of Algorithm~\ref{alg:oco_nprediction} with empirical stepsizes $\alpha = [\frac{1}{2L_g}, \frac{1}{10L_g}, \frac{1}{100L_g}, \frac{1}{1000L_g}]$ and target signals are of periodicity $10$s. In Figure (b) the dashed colored curves represent the true positions of the three targets. The solid curves represent the position estimates of the three targets at local agent $1$.}
	\label{fig:tracking_EmpStep}
\end{figure}


\subsection{Effect of prediction}
We also investigated the role of the prediction step in the empirical performance of Algorithm~\ref{alg:oco_nprediction} by comparing numerical prediction of the dynamical system based on discretization as mentioned above to the case without prediction, i.e. where $A$ in ~\eqref{eq:Algorithm_step2} is the identity matrix. The comparison was conducted by tracking targets of periodicity $10s$ with stepsize $\alpha = \frac{1}{2L_g}$. The tracking performance and dynamic regret are presented in Figure~\ref{fig:tracking_WtPredict}. In Figure~\ref{fig:tracking_WtPredict}(a), we observe that Algorithm~\ref{alg:oco_nprediction} that incorporating numerical prediction can greatly improve the dynamic regret of an algorithm compared to not incorporating any prediction (note that the figure is in log-scale), though note that the shape of both cumulative regret curves are roughly the same. This improvement is further corroborated in Figures~\ref{fig:tracking_WtPredict} (b) and (c); the dashed curves indicating the actual optimal points are generally overlapped by the solid curves of the estimated positions when prediction is utilized, whereas there is little overlap if no prediction is used.

\subsection{Comparison with existing algorithms}
We now compare the performance Algorithm~\ref{alg:oco_nprediction} to the online distributed gradient (ODG) algorithm studied in \cite{shahrampour2018distributed}. Both algorithms are implemented with prediction. The regrets achieved by these algorithms under different stepsizes are presented in Figure~\ref{fig:Compare2Ali}. Though both algorithms have comparable performance for the smaller step sizes, the behavior is wildly different for the larger step sizes. While ODG diverges for stepsizes $\frac{1}{L_g}$ and $\frac{2}{L_g}$, Algorithm~\ref{alg:oco_nprediction} is stable using these stepsizes and can further improve the accumulated regret compared to its implementation with small stepsizes. This behavior was also observed in multiple other randomly generated examples, and is consistent with that observed in \cite{ZhangCDC2019_OnlineOptimization}.

\begin{figure}
	\centering
	\subfigure[Regret of Algorithm~\ref{alg:oco_nprediction} with or without prediction]{\includegraphics[width=.8\columnwidth]{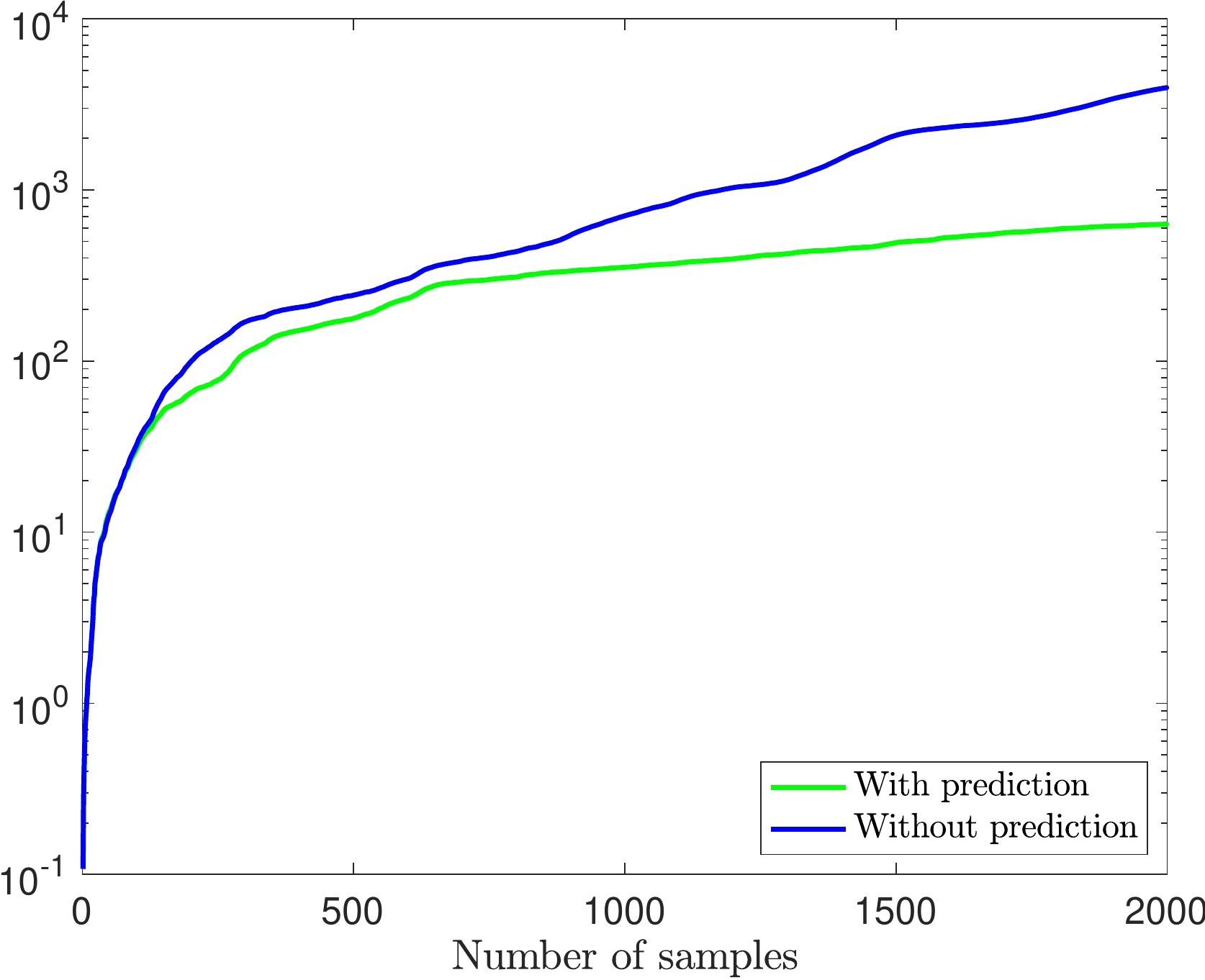}}
	\subfigure[Tracking performance without prediction]{\includegraphics[width=.45\columnwidth]{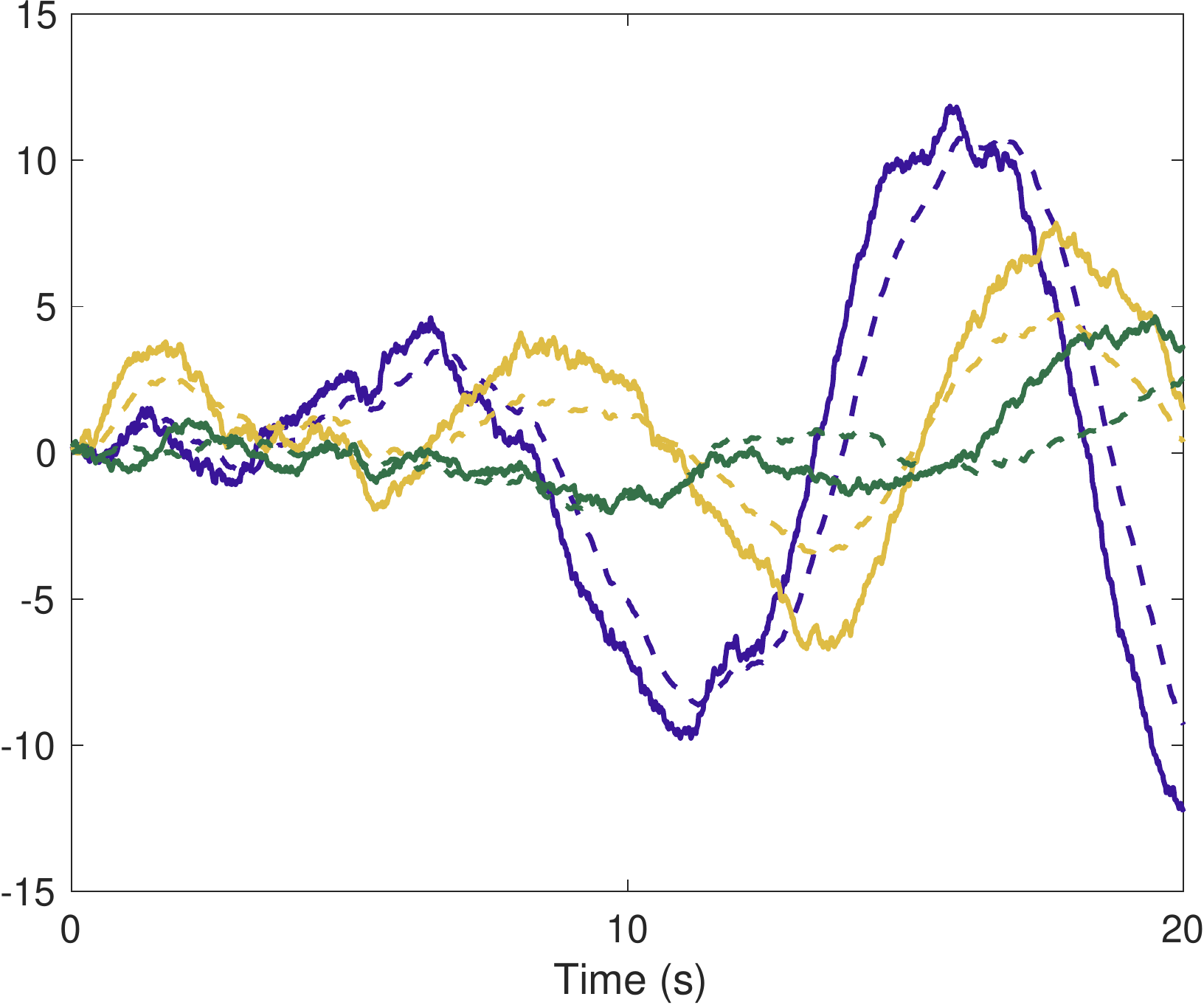}}
	\subfigure[Tracking performance with prediction]{\includegraphics[width=.45\columnwidth]{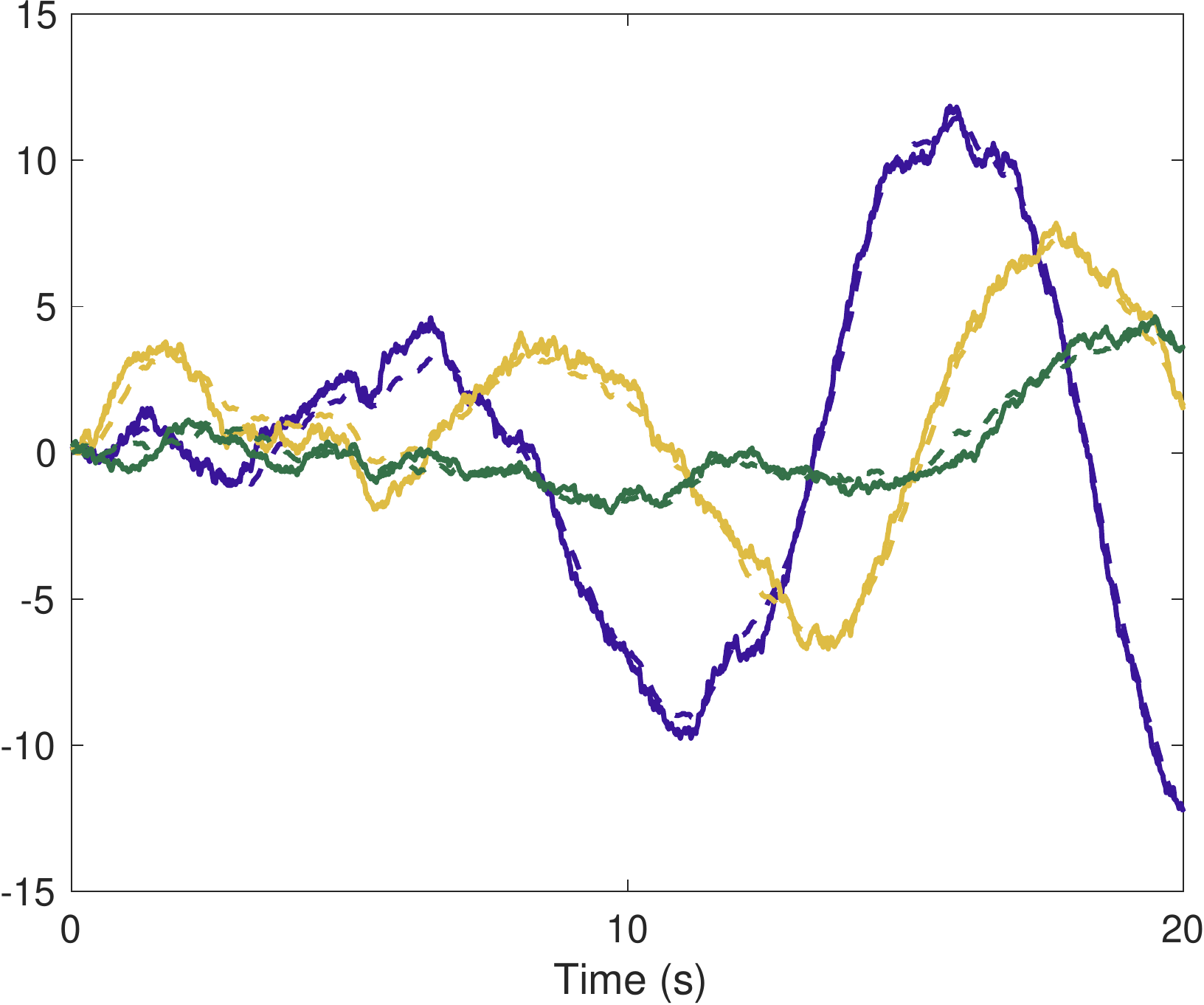}}				
	\caption{Comparison of regret and tracking performance of Algorithm~\ref{alg:oco_nprediction} with and without prediction, stepsizes $\alpha = \frac{1}{2L_g}$ and target signals are of periodicity $10$s.}
	\label{fig:tracking_WtPredict}
\end{figure}

\subsection{Necessity of the gradient path-length regularity}
Since the regret bound in \cite{shahrampour2018distributed} is only related to the optimizer path-length regularity term $\mathcal{P}_T^A$, it is natural to ask whether the dependence of the theoretical regret bound of Algorithm~\ref{alg:oco_nprediction} on $\mathcal{V}_{T}$ can be removed. In the following, we empirically show that the regret bound $R_T^d$ of using Algorithm~\ref{alg:oco_nprediction} not only depends on $\mathcal{P}_T^A$ in general, but also on $\mathcal{V}_{T}.$ To do this, we construct a numerical example for which $\mathcal{P}_T^A = 0$ for all $T$ but $R_T^d$ grows at the same speed as the regularity term $\mathcal{V}_T$. 

Consider the distributed estimation problem shown in Figure~\ref{fig:RotatingRods}, where two targets are moving along circular paths around the origin. Let $\theta_1(t)$ and $\theta_2(t)$ denote their angles from the horizontal axis, and let $R_{1}$ and $R_{2}$ denote the distances of the targets with from the origin. Four sensors need to estimate these distances, though each sensor can only measure the projected coordinate of one target onto one axis. Letting $y_t$ denote the collected measurement at time $t,$ we assume the $y_{t}$ received are of the following form

\begin{equation}
\label{eqn:MeasurementModel}
\begin{split}
y_t &= \begin{bmatrix}
\cos \theta_1(t) & 0 \\ \cos \theta_2(t) & 0 \\ 0 & \sin \theta_1(t) \\ 0 & \sin \theta_2(t)
\end{bmatrix}
\begin{bmatrix}
R_1 \\ R_2
\end{bmatrix} + v_t \\
&= C_t \; \begin{bmatrix}
R_1 \\ R_2
\end{bmatrix} + v_t,
\end{split}
\end{equation}
where each entry of $y_t$ signifies one sensor's measurement and $v_t$ is a noise term at time $t$ to be specified.
We assume that $R_1$ and $R_2$ are constant, so that the matrix governing the dynamical system in Assumption~\ref{assum:DynamicalMatrix} is the identity matrix. We also assume that the four sensors are connected in a cyclic graph. We wish to investigate the performance of Algorithm~\ref{alg:oco_nprediction} in tracking the optimizer of the global estimation problem

\begin{equation}
\label{eqn:GlobalProblem}
\min_{x_t} \; f_t(x_t) := \frac{1}{2} \sum_{i=1}^{4} \| [C_t]_i x_t - [y_t]_i \|^2
\end{equation}
\noindent where $[C_t]_i$ denotes the $i$th row of the matrix $C_t$.

We specifically design the noise terms $v_t$ in the measurement model \eqref{eqn:MeasurementModel} so that the optimizer $x_t^\ast$ of the objective is constant for all $t$, which implies that $\mathcal{P}_{T}^{A} = 0.$ To achieve this, observe that we can write the first order optimality condition as

\begin{equation}
\label{eqn:FirstOrderOptimality}
\nabla f_t(x_t) |_{x_t = x_t^\ast} = C_t^T (C_t x_t^\ast - y_t) = 0.
\end{equation}
Replacing $y_t$ in \eqref{eqn:FirstOrderOptimality} with \eqref{eqn:MeasurementModel}, we have that

\begin{equation}
\label{eqn:Derivation_1}
C_t^T(C_t x_t^\ast - C_t \begin{bmatrix}
R_1 \\ R_2
\end{bmatrix} - v_t) = 0.
\end{equation}
Thus, to ensure that $x_t^\ast = [R_1, R_2]^T$ for all $t,$ we see that $v_t$ must lie in the kernel of the matrix $C_t^T$. 

In our simulation, we let the vector $v_t$ be a unit vector in the kernel space of $C_t^T$, guaranteeing that $x_t^\ast = [R_1, R_2]^T$ for every $t$ and $\mathcal{P}_T^{A} = 0$ for all $T$ as mentioned above. However, the regularity term $\mathcal{V}_T^A$ is nonzero and grows with time because 

$$\nabla f_{i,t}(x_t^\ast) = [C_t]_i^T([C_t]_ix_t^\ast - [y_t]_i) = [v_t]_i [C_t]_i^T$$ so 

\begin{align*}
&\|g_{t+1}(Ax_{t}^\ast) - g_{t}(x_{t}^\ast)\| \\
&= \sum_{i=1}^{4}\|[v_{t+1}]_i [C_{t+1}]_i^T - [v_t]_i [C_t]_i^T\|^2 > 0
\end{align*} 
for all $t$. 

\begin{figure}[t]
	\centering
	\includegraphics[width=.8\columnwidth]{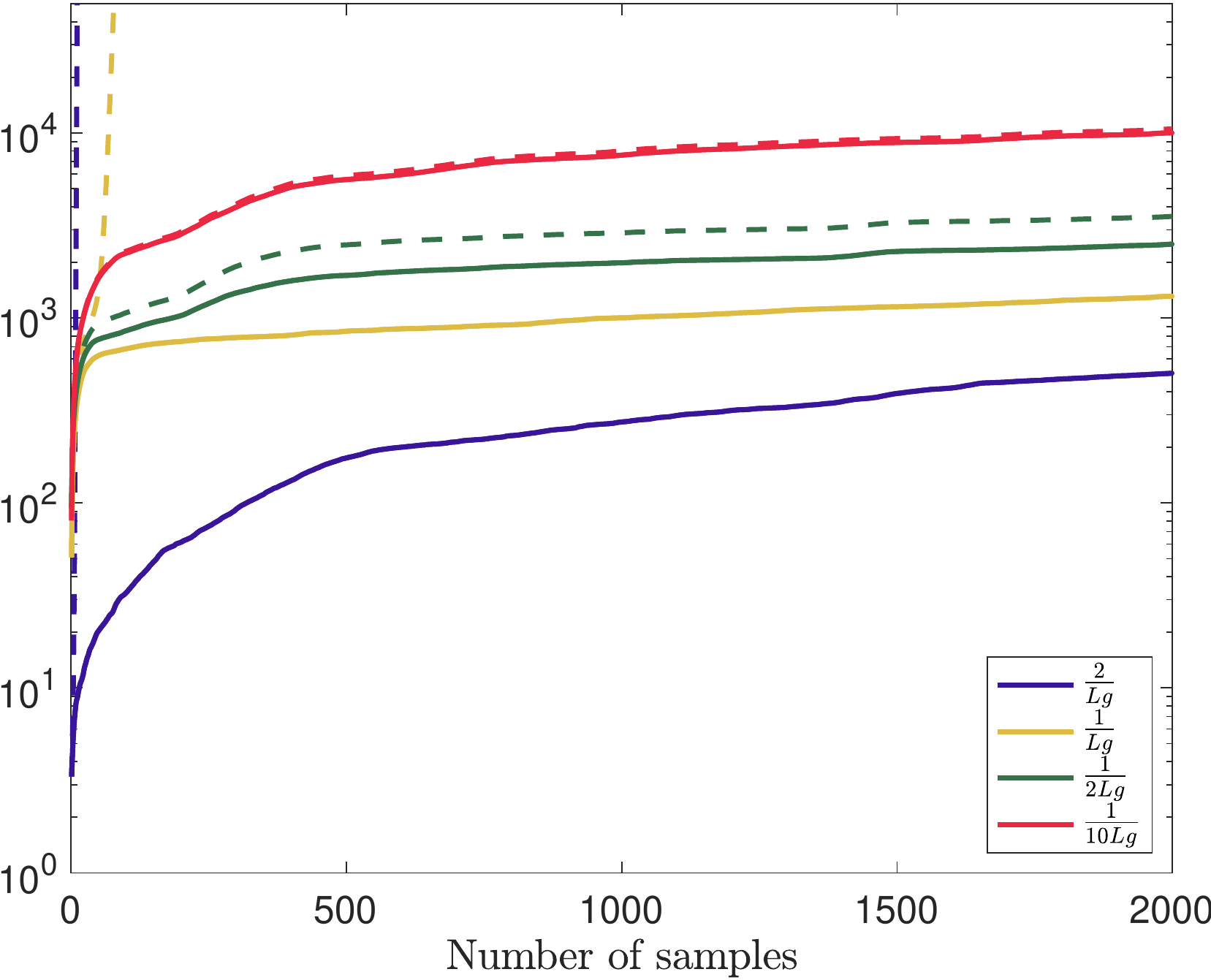}
	\caption{Comparison of the regrets between Algorithm~\ref{alg:oco_nprediction} (solid lines) and the online distributed gradient algorithm in \cite{shahrampour2018distributed} (dashed lines). Both algorithms are run with stepsizes $\alpha = [\frac{2}{Lg}, \frac{1}{Lg}, \frac{1}{2Lg}, \frac{1}{10Lg}]$.}
	\label{fig:Compare2Ali}
\end{figure}

\begin{figure}[t]
	\centering
	\includegraphics[width=.5\columnwidth]{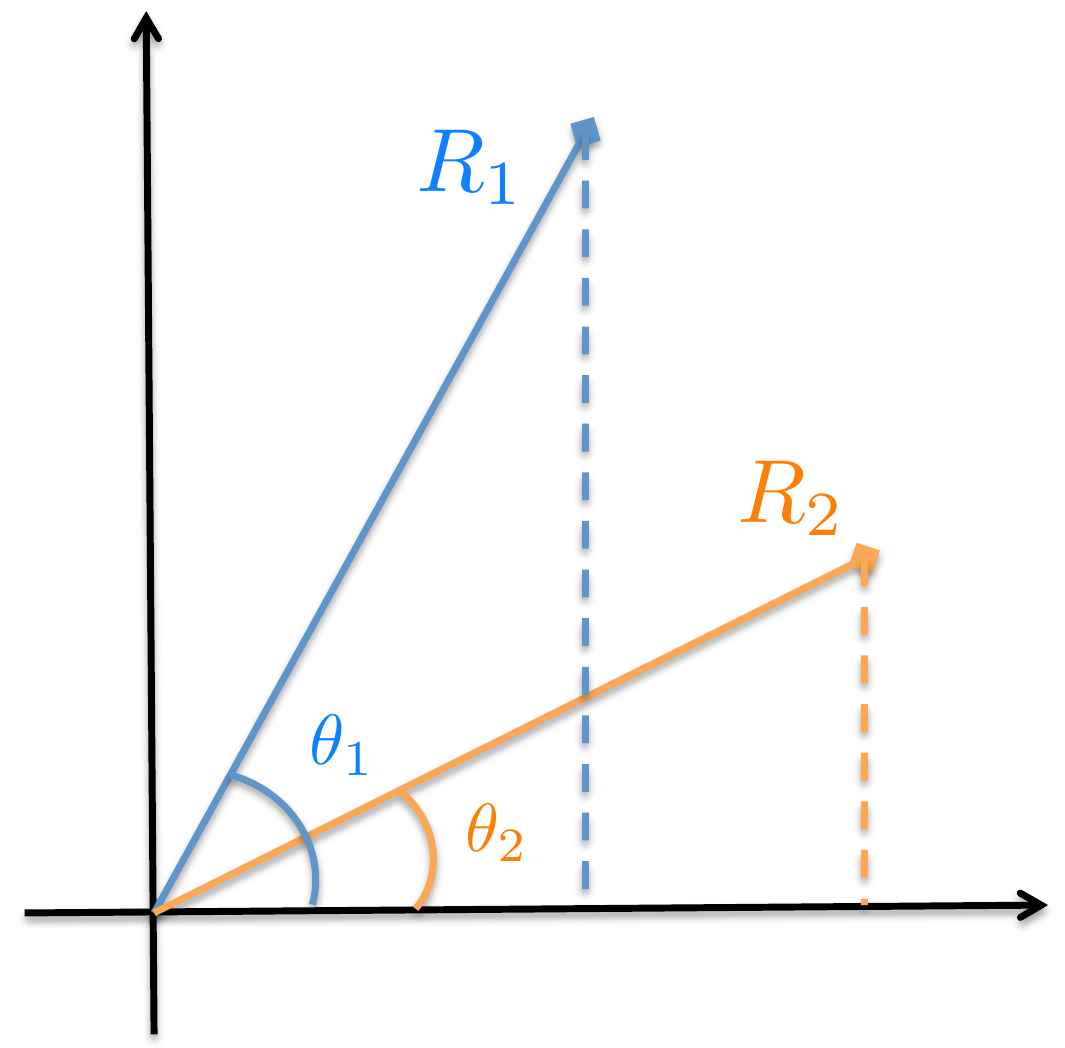} 
	\caption{Estimation of the lengths of two rotating rods}
	\label{fig:RotatingRods}
\end{figure}

We run both Algorithm~\ref{alg:oco_nprediction} and ODG to track the optimizer of problem~\eqref{eqn:GlobalProblem} for $T = 1\times10^4$ time steps. The regret curves together with the regularity curve of $\mathcal{V}_T$ are shown in Figure~\ref{fig:Regret_Regularity}. It is clear that the regret $R_d^T$ grows with the same rate as the regularity term $\mathcal{V}_T$. Similar behavior was noticed when increasing $T$ to $1 \times 10^6.$ This is because of the gradient estimator $y_{i,t}$ in \eqref{eq:Algorithm_step3} in Algorithm~\ref{alg:oco_nprediction}. Each local $y_{i, t}$ estimates the summation of the local gradients $\sum_i \nabla f_{i,t}(x_{i,t})$. However, even when this term equals 0,  if $ \nabla f_{i, t+1}(x_{i,t+1}) - \nabla f_{i, t}(x_{i,t}) \neq 0$ for all $t$, the consensus estimator $y_{i,t}$ is continuously perturbed away from the correct estimate. This perturbation error is accumulated in the regret through the gradient update in \eqref{eq:Algorithm_step1}. 

It is interesting to note that even though ODG does not apply the gradient estimator $y_{i,t}$ in Algorithm~\ref{alg:oco_nprediction}, it also suffers from such perturbations as shown in Figure~\ref{fig:Regret_Regularity}. It is most likely that the error from the perturbations is accumulated directly in the local gradient calculation, though we will not investigate this further and will leave this for future work.

\begin{figure}[t]
	\centering
	\includegraphics[width = 0.8\columnwidth]{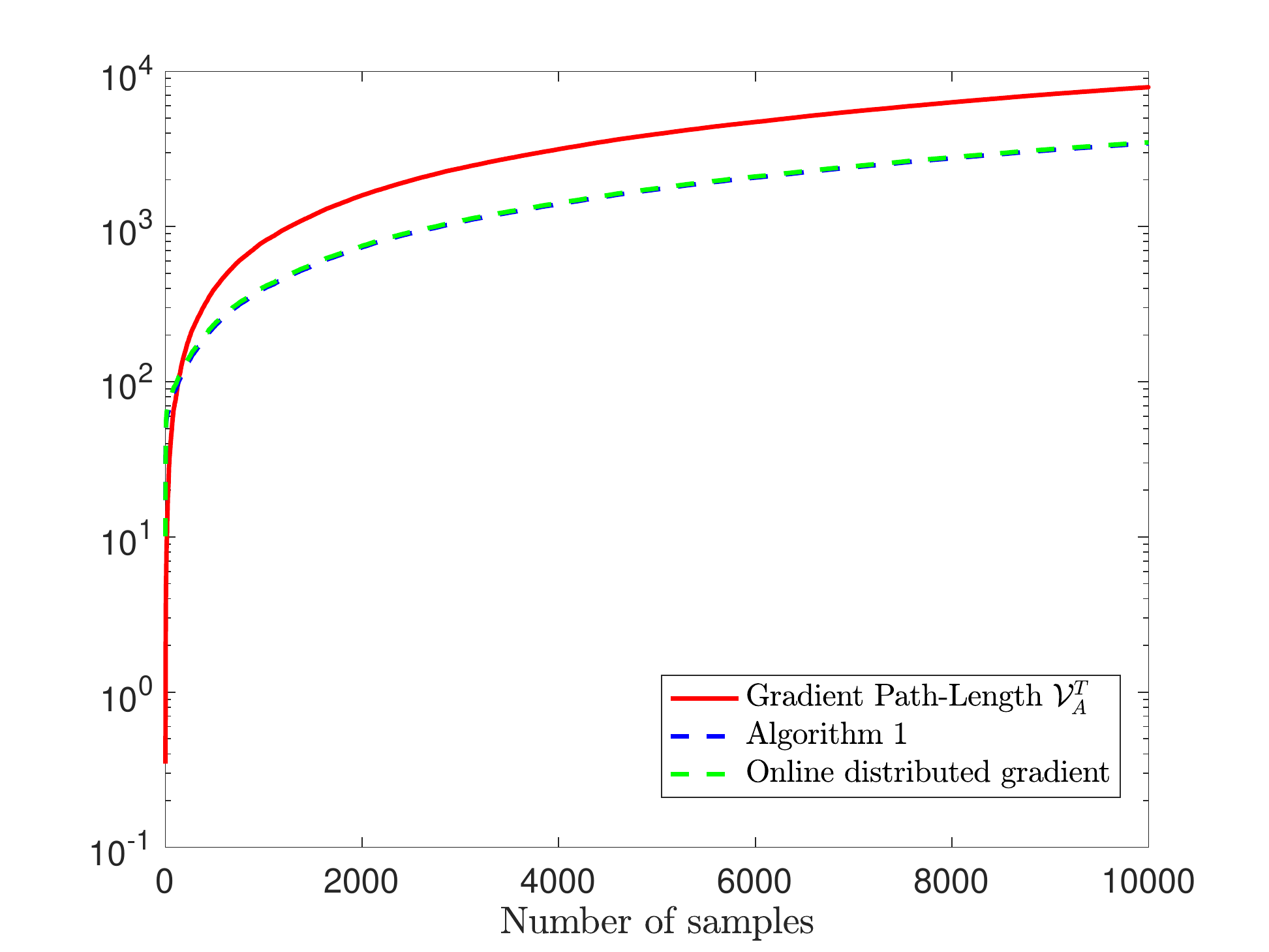}
	\caption{The accumulation of the regret $R_d^T$ (green line) and the gradient path-length regularity $\mathcal{V}_T$ (red line) by running Algorithm~\ref{alg:oco_nprediction} to solve problem~\eqref{eqn:GlobalProblem}.}
	\label{fig:Regret_Regularity}	
\end{figure}

%% file: tex/Conclusion.tex
In this paper, we proposed a novel algorithm to solve distributed online convex optimization problems. 
We adapted a gradient tracking step previously studied in static optimization methods to the online setting. When the global objective function is strongly convex, we showed that the dynamic regret of the proposed algorithm is upper bounded by a quantity that does not specifically depend on the problem horizon. This bound is tighter than existing bounds especially for long horizon problems. In addition, we proposed a new regularity measure for the time-varying optimization problem, the accumulated gradient variation at the optimal points, which is tighter than the accumulated gradient variation over the whole domain used in existing online optimization studies. We also showed that the tracking error of the proposed algorithm is asymptotically bounded. 
We evaluated the performance of our algorithms using extensive numerical experiments on distributed tracking problems, showed that our algorithm is more robust to choice of step size than that of \cite{shahrampour2018distributed}, and gave a numerical example that gives empirical evidence suggesting that the novel gradient variation measure used in our regret bound cannot in general be removed.


%% file: tex/Appendix.tex
{\bf Lemma A.1.} Let Assumption~\ref{assum:DoublyStochasticity} hold and select the stepsize $\alpha$ as
\begin{align}
	\alpha \leq \min\{ \frac{n}{\mu}, \frac{(1 - \sigma_W) (1 - \sqrt{\sigma_W})}{\sigma_W (2 + \sigma_W) + 3\sigma_W \frac{n}{\mu} L_g} \frac{1}{L_g} \}. \nonumber
\end{align}
Then, the spectral radius of the matrix $\Phi(\alpha)$ satisfies $\rho(\Phi(\alpha)) < 1$. Furthermore, all the entries in the matrix $\big( I - \Phi(\alpha)\big)^{-1}$ are upper bounded by $\frac{1}{C_{inv}} \max\{c_{ij}\}$, where 

\begin{align}
	C_{inv} = &\; \frac{\mu}{n}(1-\sigma_W)(1-\sqrt{\sigma_W}) \alpha - \frac{\mu}{n}\sigma_W(2 + \sigma_W) L_g \alpha^2 \nonumber \\
	& - 3\sigma_W L_g^2 \alpha^2, \nonumber
\end{align}
and $c_{ij}$'s are listed in \eqref{eqn:Alpha}.

\begin{proof}
	Using the formula for the inversion of a $3 \times 3$, we get that
	\begin{align}
		(I - \Phi(\alpha))^{-1} = \frac{1}{\text{det}(I - \Phi(\alpha))} C,
	\end{align}
	where $C = [c_{ij}] \in \mathbb{R}^{3 \times 3}$ has entries
	
	\begin{align}\label{eqn:Alpha}
		c_{11} &= (1 - \sqrt{\sigma_W}) (1-\sigma_W) - \sigma_W(2 + \sigma_W) L_g \alpha, \nonumber \\
		c_{12} &= \frac{1}{\sqrt{n}} (1 - \sqrt{\sigma_W}) L_g \alpha, \nonumber \\
		c_{13} &= \frac{1}{\sqrt{n}} \sigma_W L_g \alpha, \nonumber \\
		c_{21} &= 3\sigma_W L_g \alpha, \nonumber \\
		c_{22} &= \frac{\mu}{n} (1 - \sqrt{\sigma_W}) \alpha, \nonumber \\
		c_{23} &= \frac{\mu}{n} \sigma_W \alpha \nonumber \\
		c_{31} &= 3\sqrt{n} (1 - \sigma_W) L_g \alpha, \nonumber \\
		c_{32} &= 3L_g^2 \alpha^2 + \frac{\mu}{n}(2 + \sigma_W) L_g \alpha^2, \nonumber \\
		c_{33} &= \frac{\mu}{n}(1 - \sigma_W) \alpha,
	\end{align}
	and
	
	\begin{align}
		\text{det}(I - \Phi(\alpha)) & = \frac{\mu}{n}(1-\sigma_W)(1-\sqrt{\sigma_W}) \alpha \nonumber \\
		& - \frac{\mu}{n}\sigma_W(2 + \sigma_W) L_g \alpha^2 - 3\sigma_W L_g^2 \alpha^2.
	\end{align}
	Sicne the matrix $\Phi(\alpha)$ is an irreducible $3\times 3$ matrix and all the diagonal entries are strictly smaller than 1, by Lemma 5 in \cite{pu2020distributed},  to guarantee that $\rho(\Phi(\alpha)) < 1$, it is sufficient to let all entries in $\Phi(\alpha)$ and $\text{det}\big( I - \Phi(\alpha)\big)$ greater than $0$. 
	Therefore, we need to select $\alpha$ so that $1 - \frac{\mu}{n} \alpha > 0$, $c_{11} > 0$ and $\text{det}(I - \Phi(\alpha)) > 0$. Solving these inequalities, we get that
	
	\begin{align}
		\alpha \leq \min\bigl\{ \frac{n}{\mu}, \frac{(1 - \sigma_W) (1 - \sqrt{\sigma_W})}{\sigma_W (2 + \sigma_W) + 3\sigma_W \frac{n}{\mu} L_g} \frac{1}{L_g} \bigr\}.
	\end{align}
	
\end{proof}